\documentclass[a4paper,oneside,10pt]{article}%
\usepackage{amsmath}
\usepackage{amsfonts}
\usepackage{amssymb}
\usepackage{graphicx}
\usepackage{color}
\usepackage{algorithm}
\usepackage{algorithmic}
\usepackage[square,numbers,sort&compress]{natbib}%
\providecommand{\U}[1]{\protect \rule{.1in}{.1in}}
\pdfoutput=1

\newtheorem{theorem}{Theorem}[section]

\newtheorem{example}[theorem]{Example}

\newtheorem{lemma}[theorem]{Lemma}

\newtheorem{remark}[theorem]{Remark}

\newtheorem{sch}[theorem]{Scheme}
\newenvironment{proof}[1][Proof]{\noindent \textbf{#1.} }{\  \rule{0.5em}{0.5em}}
\numberwithin{equation}{section}
\begin{document}

\title{An Effective Discrete Recursive Method for Stochastic Optimal Control Problems}
\author{Mingshang Hu \thanks{Zhongtai Securities Institute for Financial Studies,
Shandong University, Jinan, Shandong 250100, PR China. humingshang@sdu.edu.cn.
Research supported by National Key R\&D Program of China (No. 2018YFA0703900)
and NSF (No. 11671231). }
\and Lianzi Jiang\thanks{Zhongtai Securities Institute for Financial Studies,
Shandong University, Jinan, Shandong 250100, PR China. jianglianzi95@163.com.}}
\date{}
\maketitle

\textbf{Abstract}. In this paper, we study the numerical method for stochastic
optimal control problems (SOCPs). By reducing the optimal control problem to
the discrete case, we derive a discrete stochastic maximum principle (SMP).
With the help of this SMP, we propose an effective discrete recursive method
for SOCPs with feedback control. We rigorously analyze errors of the proposed
method and prove that the cost obtained by our method is of first-order
convergence. Numerical experiments are carried out to support our theoretical results.



\section{Introduction}

Let $(\Omega,\mathcal{F},\{ \mathcal{F}_{t}\}_{0\leq t\leq T},P)$ be a
complete filtered probability space, on which a $d$-dimensional standard
Brownian motion $W_{t}=\left(  W_{t}^{1},\ldots,W_{t}^{d}\right)  ^{\top}$ is
given. Consider the following stochastic control system:%
\begin{equation}
\left \{
\begin{array}
[c]{l}%
dX_{t}=b\left(  t,X_{t},u_{t}\right)  dt+\sigma \left(  t,X_{t},u_{t}\right)
dW_{t},\\
X_{0}=x_{0}\in \mathbb{R}^{n},
\end{array}
\right.  \label{SDE}%
\end{equation}
with a cost functional%
\begin{equation}
J\left(  u\right)  =\mathbb{E}\left[  \int_{0}^{T}f\left(  t,X_{t}%
,u_{t}\right)  dt+h\left(  X_{T}\right)  \right]  . \label{2}%
\end{equation}
Here, $u_{\cdot}$ is the control variable valued in a convex subset
$U\subset \mathbb{R}^{m}$, $X_{\cdot}$ is the state process, and $b:[0,T]\times
\mathbb{R}^{n}\times U\rightarrow \mathbb{R}^{n},\sigma:[0,T]\times
\mathbb{R}^{n}\times U\rightarrow \mathbb{R}^{n\times d}$, $f:$ $[0,T]\times
\mathbb{R}^{n}\times U\rightarrow \mathbb{R}$ and $h:\mathbb{R}^{n}%
\rightarrow \mathbb{R}$ are given functions.

An admissible control $u_{\cdot}$ is an $\{ \mathcal{F}_{t}\}_{0\leq t\leq T}%
$-adapted process with values in $U$ such that
\[
\mathbb{E}\left[  \int_{0}^{T}|u_{t}|^{2}dt\right]  <\infty.
\]
The set of admissible controls is denoted by $\mathcal{U}[0,T]$. Our
stochastic optimal control problem (SOCP) is to find a control $u_{\cdot
}^{\ast}\in \mathcal{U}[0,T]$ such that
\begin{equation}
J\left(  u^{\ast}\right)  =\min_{u_{\cdot}\in \mathcal{U}[0,T]}J\left(
u\right)  . \label{3}%
\end{equation}
The process $u_{\cdot}^{\ast}\ $is called an optimal control. The state
process $X_{\cdot}^{\ast}$ corresponding to $u_{\cdot}^{\ast}$ is called an
optimal state process, and $\left(  X_{\cdot}^{\ast},u_{\cdot}^{\ast}\right)
$ is called an optimal pair.

In practice, the control usually depends on the historical information of\ the
state process.$\ $For example, in option pricing and portfolio optimization,
people make current decisions based on the historical stock price information.
It is worth noting that the most important control of this type is feedback
control, that is, the control is given by the current state. More precisely,
there exists a function $\phi$ such that $u_{t}=\phi \left(  t,X_{t}\right)  $
(see \cite{YZh1999}). By observing the current state information, feedback
control can be easily operated.
For this reason, we assume that the optimal control is a feedback control in
this paper.

However, the SOCP does not directly yield an explicit solution, and thus
efficient numerical methods have been widely studied in recent years. Most of
the existing numerical algorithms are based on the dynamic programming
principle (DPP) and the associated Hamilton-Jacobi-Bellman (HJB) equations
(see, e.g.,
\cite{BJ2005,HFL2012,Kush2013,KD2001,RF2016,SSul2014,WF2008,WR2011}). While
stochastic maximum principle (SMP) is a popular tool for theoretical studies
of stochastic optimal control (see, e.g.,
\cite{MSXZh2013,OS2009,P1990,PBGM1962} and the references therein), it has not
been widely used in numerical algorithms
\cite{DSL2013,FZhZh2020,GLTZhZh2017,LSHP2012}. Let us mention some recent
works \cite{LSHP2012,DSL2013}, which proposed numerical algorithms for SOCPs
based on the SMP, and their discussions are limited to the case where the
control $u_{t}$ is a deterministic function of $t$. Furthermore, by
introducing the Euler method to solve the adjoint equation, Gong et al.
\cite{GLTZhZh2017} proposed a gradient projection algorithm for SOCPs and
first obtained the rate of convergence for the deterministic control case.
Recently, in \cite{FZhZh2020}, the authors propose a numerical algorithm for
SOCPs with feedback control by means of forward backward stochastic
differential equations (FBSDEs), but the convergence is not proved theoretically.

Our main results are the following. We first reduce the optimal control
problem to the discrete case and obtain a discrete SMP. The discrete SMP
coupled with the state and adjoint equations forms a discrete Hamiltonian
system. Then we propose a discrete recursive method for SOCPs by approximating
the discrete Hamiltonian system. Considering that the goal of the SOCP is to
select an appropriate control to achieve the optimal cost, we rigorously
analyze errors of the proposed method and proved that the cost obtained by our
method is of first-order convergence. We remark that the numerical algorithm
of our discrete recursive method is consistent with the algorithm in
\cite{FZhZh2020}. Several examples are presented to support the theoretical results.

The rest of the paper is organized as follows. In Section 2, we present some
preliminaries. By establishing a discrete SMP, we propose a discrete recursive
algorithm for solving SOCPs in Section 3. In Section 4, we prove the main
convergence results. In Section 5, various numerical tests are given to
demonstrate high accuracy of our method.

\section{Preliminaries}

We recall some basic results about forward and backward stochastic
differential equations (SDEs) in this section, which can be found in
\cite{KPQ1997,MY2007,PP1990,YZh1999}. We will use the following notations:

$L_{\mathcal{F}}^{2}\left(  0,T;\mathbb{R}^{n}\right)  :$ the set of
$\mathbb{R}^{n}$-valued and $\mathcal{F}_{t}$-adapted stochastic processes
such that $\mathbb{E}[\int_{0}^{T}|\varphi_{t}|^{2}dt]<\infty$.

$C_{b}^{k}:$ the set of continuously differentiable functions $\varphi
:\mathbb{R}^{n}\rightarrow \mathbb{R}$ with uniformly bounded derivatives
$\partial_{x}^{k_{1}}\varphi$ for $k_{1}\leq k$.

$C_{b}^{l,k}:$ the set of continuously differentiable functions $\varphi
:\left[  0,T\right]  \times$ $\mathbb{R}^{n}\rightarrow \mathbb{R}$ with
uniformly bounded partial derivatives $\partial_{t}^{l_{1}}\varphi$ and
$\partial_{x}^{k_{1}}\varphi$ for $l_{1}\leq l$ and $k_{1}\leq k$.

$C_{b}^{l,k,k}:$ the set of continuously differentiable functions
$\varphi:\left[  0,T\right]  \times \mathbb{R}^{n}\times U\rightarrow
\mathbb{R}$ with uniformly bounded partial derivatives $\partial_{t}^{l_{1}%
}\varphi$ and $\partial_{x}^{k_{1}}\partial_{u}^{k_{2}}\varphi$ for $l_{1}\leq
l$ and $k_{1}+k_{2}\leq k$.

We first recall the following standard estimate of SDE.

\begin{lemma}
\label{lemma2}Let $X_{t}^{i}$, $i=1,2,$ be the solution of the following SDE:%
\[
X_{t}^{i}=X_{0}^{i}+\int_{0}^{t}b^{i}\left(  s,X_{s}^{i}\right)  ds+\int
_{0}^{t}\sigma^{i}\left(  s,X_{s}^{i}\right)  dW_{s},
\]
where $b^{i}=b^{i}\left(  s,x\right)  :[0,T]\times \mathbb{R}^{n}%
\rightarrow \mathbb{R}^{n}$ and $\sigma^{i}=\sigma^{i}\left(  s,x\right)
:[0,T]\times \mathbb{R}^{n}\rightarrow \mathbb{R}^{n\times d}$ are Lipschitz in
$x$, $b^{i}\left(  s,0\right)  \in L_{\mathcal{F}}^{2}\left(  0,T;\mathbb{R}%
^{n}\right)  $ and $\sigma^{i}\left(  s,0\right)  \in L_{\mathcal{F}}%
^{2}\left(  0,T;\mathbb{R}^{n\times d}\right)  $. Then there exists a constant
$C>0$ depending on $T$ and the Lipschitz constant such that
\[%
\begin{array}
[c]{l}%
\displaystyle \mathbb{E}\left[  \sup_{0\leq t\leq T}\left \vert X_{t}^{1}%
-X_{t}^{2}\right \vert ^{2}\right]  \leq C\mathbb{E}\left[  \left \vert
X_{0}^{1}-X_{0}^{2}\right \vert ^{2}\right] \\
\displaystyle \text{ \  \  \  \  \  \  \  \  \  \  \ }+C\int_{0}^{T}\mathbb{E}\left[
\left \vert b^{1}\left(  s,X_{s}^{1}\right)  -b^{2}\left(  s,X_{s}^{1}\right)
\right \vert ^{2}+\left \vert \sigma^{1}\left(  s,X_{s}^{1}\right)  -\sigma
^{2}\left(  s,X_{s}^{1}\right)  \right \vert ^{2}\right]  ds.
\end{array}
\]

\end{lemma}

The following lemma is the well-known Feynman-Kac formula, which gives the
stochastic representation for the solutions to some parabolic partial
differential equations (PDEs).

\begin{lemma}
\label{F-K}Assume the functions $b:[0,T]\times \mathbb{R}^{n}\rightarrow
\mathbb{R}^{n}$, $\sigma:[0,T]\times \mathbb{R}^{n}\rightarrow \mathbb{R}%
^{n\times d}$, $g:\mathbb{R}^{n}\rightarrow \mathbb{R}^{n}$ and $F:\left[
0,T\right]  \times \mathbb{R}^{n}\times \mathbb{R}^{n}\times \mathbb{R}^{n\times
d}\rightarrow \mathbb{R}^{n}$ are uniformly Lipschitz continuous w.r.t.
$\left(  x,y,z\right)  $ and continuous w.r.t. $t$,$\ $and the matrix-valued
function $a=\sigma \sigma^{\top}$ is uniformly elliptic. For any given $\left(
t,x\right)  \in \left[  0,T\right]  \times \mathbb{R}^{n}$, $\left(  Y_{\cdot
}^{t,x},Z_{\cdot}^{t,x}\right)  $ is the solution of the following FBSDEs:
\begin{equation}
\left \{
\begin{array}
[c]{ll}%
dX_{s}^{t,x}=b\left(  s,X_{s}^{t,x}\right)  ds+\sigma \left(  s,X_{s}%
^{t,x}\right)  dW_{s},\text{ \ } & s\in \left[  t,T\right]  ,\\
dY_{s}^{t,x}=-F\left(  s,X_{s}^{t,x},Y_{s}^{t,x},Z_{s}^{t,x}\right)
ds+Z_{s}^{t,x}dW_{s}, & s\in \left[  t,T\right]  ,\\
X_{t}^{t,x}=x, & \\
Y_{T}^{t,x}=g\left(  X_{T}^{t,x}\right)  . &
\end{array}
\right.
\end{equation}
Then $v\left(  t,x\right)  =Y_{t}^{t,x}\ $is a unique solution of the
following PDE:%
\begin{equation}
\left \{
\begin{array}
[c]{l}%
\mathcal{L}v\left(  t,x\right)  =-F\left(  t,x,v\left(  t,x\right)
,\sigma \left(  t,x\right)  \partial_{x}v\left(  t,x\right)  \right)  ,\\
v\left(  T,x\right)  =g\left(  x\right)  ,\  \  \  \  \  \forall \left(  t,x\right)
\in \left[  0,T\right]  \times \mathbb{R}^{n},
\end{array}
\right.  \label{f-k}%
\end{equation}
where $\mathcal{L}$ is the differential operator defined by
\[
\mathcal{L}=\frac{\partial}{\partial t}+\sum \limits_{i=1}^{n}b_{i}\left(
t,x\right)  \frac{\partial}{\partial x_{i}}+\frac{1}{2}\sum \limits_{i,j=1}%
^{n}\left[  \sigma \sigma^{\top}\right]  _{i,j}\left(  t,x\right)
\frac{\partial^{2}}{\partial x_{i}\partial x_{j}}.
\]
Furthermore, for $k=0,1,2,\ldots,$ if $b,\sigma \in C_{b}^{1+k,2+2k}$, $F\in
C_{b}^{1+k,2+2k,2+2k,2+2k}$ and $g\in$ $C_{b}^{2+2k+\alpha}$ for some
$\alpha \in \left(  0,1\right)  $, then $v\in C_{b}^{1+k,2+2k}$.
\end{lemma}

\begin{remark}
\label{remark_F-K}In the case when $F\left(  t,x,y,z\right)  \equiv0$, it is
easy to see $v\left(  t,x\right)  =\mathbb{E}\left[  g(X_{T}^{t,x})\right]  $,
and $\left(  \ref{f-k}\right)  $ reduces to%
\begin{equation}
\left \{
\begin{array}
[c]{l}%
\mathcal{L}v\left(  t,x\right)  =0,\text{ \ }\forall \left(  t,x\right)
\in \left[  0,T\right]  \times \mathbb{R}^{n},\\
v\left(  T,x\right)  =g\left(  x\right)  .
\end{array}
\right.  \label{SDE-PDE}%
\end{equation}

\end{remark}

\section{The discrete recursive method}

In this section, we propose the discrete recursive method for SOCPs. Consider
the stochastic control system $\left(  \ref{SDE}\right)  -\left(
\ref{2}\right)  $. Let $\left(  X_{\cdot}^{\ast},u_{\cdot}^{\ast}\right)  $ be
the optimal pair defined in $\left(  \ref{3}\right)  $, and assume that
$u_{t}^{\ast}=\phi^{\ast}\left(  t,X_{t}^{\ast}\right)  $, where $\phi^{\ast
}:\left[  0,T\right]  \times \mathbb{R}^{n}\rightarrow U$ is a function. For
simplicity of presentation, we suppose $d=1$. We need the following assumption.

\begin{description}
\item[$\left(  A1\right)  $] For $\varphi=b,\sigma,f,h$\ and $\phi^{\ast}$,
$\varphi$, $\varphi_{x}$, $\varphi_{u}$ are continuous in $\left(
t,x,u\right)  $ and $\varphi_{x}$, $\varphi_{u}$ are bounded.
\end{description}

For the time interval $\left[  0,T\right]  $ and a given positive integer $N$,
we use the following uniform partition:
\[
0=t_{0}<t_{1}<\cdots<t_{N}=T,
\]
with $\Delta t:=t_{i+1}-t_{i}=T/N$, and denote $\Delta W_{t_{i+1}}%
:=W_{t_{i+1}}-W_{t_{i}}$ for $0\leq i\leq N-1$. Define the piecewise
admissible control set%
\begin{equation}
\mathcal{U}^{N}[0,T]=\left \{  u_{t}=\sum_{i=0}^{N-1}\phi_{i}({X}_{t_{i}%
})I_{\left[  t_{i},t_{i+1}\right)  }(t):\phi_{i}(\cdot)\in C_{b}^{1}\left(
\mathbb{R}^{n};U\right)  \right \}  . \label{discrete control set}%
\end{equation}
This means that it takes the feedback value of the state process at $t_{i}$ as
the control on $\left[  t_{i},t_{i+1}\right]  $, and one can check that
$\mathcal{U}^{N}[0,T]\subset \mathcal{U}[0,T]$ is convex. Now we define the
discrete optimal control problem over $\mathcal{U}^{N}[0,T]$:
\begin{equation}
J\left(  \bar{u}\right)  =\min_{u_{\cdot}\in \mathcal{U}^{N}[0,T]}J\left(
u\right)  . \label{cost function}%
\end{equation}
We call $\bar{u}_{\cdot}$ a discrete optimal control, which has the following
expression%
\begin{equation}
\bar{u}_{t}=\sum_{i=0}^{N-1}\bar{\phi}_{i}(\bar{X}_{t_{i}})I_{\left[
t_{i},t_{i+1}\right)  }(t),\text{ \ }\bar{\phi}_{i}(\cdot)\in C_{b}^{1}\left(
\mathbb{R}^{n};U\right)  . \label{u}%
\end{equation}
The corresponding $\bar{X}_{\cdot}$ and $\left(  \bar{X}_{\cdot},\bar
{u}_{\cdot}\right)  $ are called a discrete optimal state process and discrete
optimal pair, respectively. We remark that $\left(  \bar{X}_{\cdot},\bar
{u}_{\cdot}\right)  $ essentially depends on the time partition $N$. For
simplicity, we omit $N$ without causing confusion. Set
\[%
\begin{array}
[c]{ll}%
b(\cdot)=(b_{1}(\cdot),\ldots,b_{n}(\cdot))^{\top}, & \sigma(\cdot
)=(\sigma_{1}(\cdot),\ldots,\sigma_{n}(\cdot))^{\top},\\
b^{i}(t)=b(t,\bar{X}_{t},\bar{\phi}_{i}(\bar{X}_{t_{i}})), & \sigma
^{i}(t)=\sigma(t,\bar{X}_{t},\bar{\phi}_{i}(\bar{X}_{t_{i}})),
\end{array}
\]
and define similarly $b_{x}^{i}(t),b_{u}^{i}(t),\sigma_{x}^{i}(t)\ $and
$b_{u}^{i}(t)$, where
\[
b_{x}(\cdot)=\left[
\begin{array}
[c]{rrr}%
b_{1x_{1}}(\cdot), & \ldots & ,b_{1x_{n}}(\cdot)\\
\multicolumn{1}{c}{\vdots} & \multicolumn{1}{c}{} & \multicolumn{1}{c}{\vdots
}\\
b_{nx_{1}}(\cdot), & \ldots & ,b_{nx_{n}}(\cdot)
\end{array}
\right]  ,
\]
and the other derivatives can be similarly defined. Under assumption $\left(
A1\right)  $, the discrete optimal state process $\bar{X}_{\cdot}$ can be
uniquely solved by the following piecewise equation:%
\begin{equation}
\left \{
\begin{array}
[c]{l}%
d\bar{X}_{t}=b^{i}(t)dt+\sigma^{i}(t)dW_{t},\text{ }t\in \lbrack t_{i}%
,t_{i+1}],\\
\bar{X}_{t_{i}}=\bar{X}_{t_{i}},\ i=0,1,\ldots,N-1,
\end{array}
\right.  \label{4}%
\end{equation}
where $\bar{X}_{0}=x_{0}$.

\subsection{Discrete stochastic maximum principle}

In this subsection, we derive a discrete SMP, which plays an important role in
the proposal of the discrete recursive method for SOCPs. For any fixed integer
$0\leq i\leq N-1$, take an arbitrary $\phi_{i}(\cdot)\in C_{b}^{1}\left(
\mathbb{R}^{n};U\right)  $. For each $0\leq \varepsilon \leq1$, we introduce
$\phi_{i}^{\varepsilon}(\cdot)=\bar{\phi}_{i}(\cdot)+\varepsilon \delta \phi
_{i}\left(  \cdot \right)  \in C_{b}^{1}\left(  \mathbb{R}^{n};U\right)  $ on
$\left[  t_{i},t_{i+1}\right]  $ with $\delta \phi_{i}(\cdot)=\phi_{i}%
(\cdot)-\bar{\phi}_{i}(\cdot)$. Denote
\[
u_{t}^{i,\varepsilon}=\sum_{j=0}^{i-1}\bar{\phi}_{j}(X_{t_{j}}^{i,\varepsilon
})I_{[t_{j},t_{j+1})}(t)+\phi_{i}^{\varepsilon}(X_{t_{i}}^{i,\varepsilon
})I_{\left[  t_{i},t_{i+1}\right)  }(t)+\sum_{j=i+1}^{N-1}\bar{\phi}%
_{j}(X_{t_{j}}^{i,\varepsilon})I_{[t_{j},t_{j+1})}(t).
\]
It is easy to see that $u_{\cdot}^{i,\varepsilon}\in \mathcal{U}^{N}[0,T]$, and
the corresponding state process $X_{t}^{i,\varepsilon}\equiv \bar{X}_{t}$ on
$t\in \lbrack0,t_{i}]$,
\begin{align}
&  \left \{
\begin{array}
[c]{l}%
dX_{t}^{i,\varepsilon}=b\left(  t,X_{t}^{i,\varepsilon},\phi_{i}^{\varepsilon
}(X_{t_{i}}^{i,\varepsilon})\right)  dt+\sigma \left(  t,X_{t}^{i,\varepsilon
},\phi_{i}^{\varepsilon}(X_{t_{i}}^{i,\varepsilon})\right)  dW_{t},\\
X_{t_{i}}^{i,\varepsilon}=\bar{X}_{t_{i}},\text{ \ }t\in \lbrack t_{i}%
,t_{i+1}],
\end{array}
\right.  \text{\  \ }\label{5.1}\\
&  \left \{
\begin{array}
[c]{l}%
dX_{t}^{i,\varepsilon}=b\left(  t,X_{t}^{i,\varepsilon},\bar{\phi}%
_{j}(X_{t_{j}}^{i,\varepsilon})\right)  dt+\sigma \left(  t,X_{t}%
^{i,\varepsilon},\bar{\phi}_{j}(X_{t_{j}}^{i,\varepsilon})\right)
dW_{t},\text{ \ }\\
X_{t_{j}}^{i,\varepsilon}=X_{t_{j}}^{i,\varepsilon},\text{ \ }t\in \lbrack
t_{j},t_{j+1}],\text{ \ }j=i+1,\ldots,N-1.
\end{array}
\right.  \label{5.2}%
\end{align}
The variational equation $\hat{X}_{\cdot}^{i}$ can be given as follows:
$\hat{X}_{t}^{i}\equiv0$ on $t\in \lbrack0,t_{i}]$,
\begin{align}
&  \left \{
\begin{array}
[c]{l}%
d\hat{X}_{t}^{i}=\left[  b_{x}^{i}(t)\hat{X}_{t}^{i}+b_{u}^{i}(t)\delta
\phi_{i}(\bar{X}_{t_{i}})\right]  dt\\
\text{ \  \  \  \  \  \  \  \  \  \  \  \  \  \  \  \  \  \ }+\left[  \sigma_{x}^{i}(t)\hat
{X}_{t}^{i}+\sigma_{u}^{i}(t)\delta \phi_{i}(\bar{X}_{t_{i}})\right]  dW_{t},\\
\hat{X}_{t_{i}}^{i}=0,\text{\  \ }t\in \lbrack t_{i},t_{i+1}],
\end{array}
\right.  \text{\  \  \  \  \  \  \  \  \  \  \  \  \  \ }\label{variation equation}\\
&  \left \{
\begin{array}
[c]{l}%
d\hat{X}_{t}^{i}=\left[  b_{x}^{j}(t)\hat{X}_{t}^{i}+b_{u}^{j}(t)\bar{\phi
}_{j,x}(\bar{X}_{t_{j}})\hat{X}_{t_{j}}^{i}\right]  dt\\
\text{ \  \  \  \  \  \  \  \  \  \  \  \  \  \ }+\left[  \sigma_{x}^{j}(t)\hat{X}_{t}%
^{i}+\sigma_{u}^{j}(t)\bar{\phi}_{j,x}(\bar{X}_{t_{j}})\hat{X}_{t_{j}}%
^{i}\right]  dW_{t},\\
\hat{X}_{t_{j}}^{i}=\hat{X}_{t_{j}}^{i},\text{ \ }t\in \lbrack t_{j}%
,t_{j+1}],\text{ }j=i+1,\ldots,N-1.
\end{array}
\right.  \text{ \ } \label{variation equation2}%
\end{align}

We also introduce the following adjoint equation:
\begin{equation}
\left \{
\begin{array}
[c]{l}%
-d\bar{P}_{t}=H_{x}\left(  \bar{X}_{t},\bar{P}_{t},\bar{Q}_{t},\bar{\phi}%
_{i}(\bar{X}_{t_{i}})\right)  dt-\bar{Q}_{t}dW_{t},\\
\bar{P}_{t_{i+1}}=\bar{P}_{t_{i+1}},\text{\ }t\in \lbrack t_{i},t_{i+1}%
],\text{\ }i=0,1,\ldots,N-1,
\end{array}
\right.  \text{ \  \  \  \  \  \  \  \  \  \  \  \  \  \  \  \ } \label{adjoint equation}%
\end{equation}
with $\bar{P}_{T}=h_{x}(\bar{X}_{T})$, where the Hamiltonian $H:\left[
0,T\right]  \times \mathbb{R}^{n}\times \mathbb{R}^{n}\times \mathbb{R}^{n}\times
U\rightarrow \mathbb{R}$ is defined as follows:%
\[
H\left(  t,x,p,q,u\right)  =\left \langle p,b\left(  t,x,u\right)
\right \rangle +\left \langle q,\sigma \left(  t,x,u\right)  \right \rangle
+f\left(  t,x,u\right)  ,
\]
and denote
\[%
\begin{array}
[c]{r}%
H_{x}(\cdot)=(H_{x_{1}}(\cdot),\ldots,H_{x_{n}}(\cdot))^{\top},H_{u}%
(\cdot)=(H_{u_{1}}(\cdot),\ldots,H_{u_{m}}(\cdot))^{\top},h_{x}(\cdot
)=(h_{x_{1}}(\cdot),\ldots,h_{x_{n}}(\cdot))^{\top}.
\end{array}
\]

Now we establish the following discrete SMP.

\begin{theorem}
\label{theorem1}Suppose $\left(  A1\right)  $ holds. Let $\left(  \bar
{X}_{\cdot},\bar{u}_{\cdot}\right)  $ be the discrete optimal pair of the
problem $\left(  \ref{cost function}\right)  $, and let $\left(  \bar
{P}_{\cdot},\bar{Q}_{\cdot}\right)  $ be the solution to $\left(
\ref{adjoint equation}\right)  $. Then for $i=N-1,\ldots,1,0$,
\begin{equation}
\mathbb{E}\left[  \left \langle \int_{t_{i}}^{t_{i+1}}H_{u}\left(  t,\bar
{X}_{t},\bar{P}_{t},\bar{Q}_{t},\bar{\phi}_{i}(\bar{X}_{t_{i}})\right)
dt,\phi_{i}(\bar{X}_{t_{i}})-\bar{\phi}_{i}(\bar{X}_{t_{i}})\right \rangle
\right]  \geq0,\text{ \ }\forall \phi_{i}(\cdot)\in C_{b}^{1}(\mathbb{R}%
^{n};U). \label{Theorem 1.2}%
\end{equation}
Furthermore, if $\bar{\phi}_{i}(x)$ is an interior point of $U$, for
$x\in \mathbb{R}^{n}$, then
\begin{equation}
\mathbb{E}\left[  \left.  \int_{t_{i}}^{t_{i+1}}H_{u}\left(  t,\bar{X}%
_{t},\bar{P}_{t},\bar{Q}_{t},\bar{\phi}_{i}(x)\right)  dt\right \vert \bar
{X}_{t_{i}}=x\right]  =0,\text{ }P_{\bar{X}_{t_{i}}}\text{-a.s.}
\label{Theorem 1}%
\end{equation}

\end{theorem}

In order to prove Theorem \ref{theorem1}, we need the following lemmas.

\begin{lemma}
\label{lemma1}Suppose $\left(  A1\right)  $ holds. Then for any integer $0\leq
i\leq N-1$,%
\begin{equation}
\lim_{\varepsilon \downarrow0}\sup_{t\in \lbrack0,T]}\mathbb{E}\left[
|\tilde{X}_{t}^{i,\varepsilon}|^{2}\right]  =0, \label{I}%
\end{equation}
where%
\[
\tilde{X}_{t}^{i,\varepsilon}=\varepsilon^{-1}[X_{t}^{i,\varepsilon}-\bar
{X}_{t}]-\hat{X}_{t}^{i}.
\]

\end{lemma}

\begin{proof}
For the proof of lemma, one can refer to Lemma 4.1 of \cite{B1982I}.
\end{proof}

\begin{lemma}
\label{relation lemma}Suppose that the conditions in Theorem \ref{theorem1}
hold. The value function \
\[
V^{i}(x) :=\mathbb{E}\left[  \int_{t_{i}}^{t_{i+1}}f\left(  t,\bar{X}%
_{t}^{t_{i},x},\bar{\phi}_{i}(x) \right)  dt+V^{i+1}(\bar{X}_{t_{i+1}}%
^{t_{i},x})\right]  ,\text{ }x\in \mathbb{R}^{n}\text{, }i=N-1,\ldots,1,0,
\]
with $V^{N}(x) =h(x) $, where $\bar{X}_{\cdot}^{t_{i},x}$ is the solution of
$\left(  \ref{4}\right)  $ starting from $\left(  t_{i},x\right)  $, and
$(\bar{P}_{\cdot}^{t_{i},x},\bar{Q}_{\cdot}^{t_{i},x})$ is the solution of
$\left(  \ref{adjoint equation}\right)  $ related to $\bar{X}_{\cdot}%
^{t_{i},x}$. Assume $V_{x}^{i+1}(x) =\bar{P}_{t_{i+1}}^{t_{i+1},x}$, for some
integer $0\leq i\leq N-1$. Then $V_{x}^{i}(x) =\bar{P}_{t_{i}}^{t_{i},x}$ if
and only if
\begin{equation}
\mathbb{E}\left[  \int_{t_{i}}^{t_{i+1}}\left(  \bar{\phi}_{i,x}\left(
x\right)  \right)  ^{\top}H_{u}\left(  t,\bar{X}_{t}^{t_{i},x},\bar{P}%
_{t}^{t_{i},x},\bar{Q}_{t}^{t_{i},x},\bar{\phi}_{i}(x) \right)  dt\right]  =0.
\label{18}%
\end{equation}

\begin{proof}
For any integer $i$, we define
\[%
\begin{array}
[c]{ll}%
f_{x}^{i}(t)=f_{x}\left(  t,\bar{X}_{t}^{t_{i},x},\bar{\phi}_{i}(\bar
{X}_{t_{i}}^{t_{i},x})\right)  , & f_{u}^{i}(t)=f_{u}\left(  t,\bar{X}%
_{t}^{t_{i},x},\bar{\phi}_{i}(\bar{X}_{t_{i}}^{t_{i},x})\right)  ,
\end{array}
\]
where $f_{x}^{i}(\cdot)=\left(  f_{x_{1}}^{i}\left(  \cdot \right)
,\ldots,f_{x_{n}}^{i}(\cdot)\right)  ^{\top}$ and $f_{u}^{i}=\left(  f_{u_{1}%
}^{i}(\cdot),\ldots,f_{u_{m}}^{i}(\cdot)\right)  ^{\top}$. By the classical
variational method, one can check that
\begin{equation}
V_{x}^{i}(x)=\mathbb{E}\left[  \int_{t_{i}}^{t_{i+1}}\left(  (\check{X}%
_{t}^{t_{i},x})^{\top}f_{x}^{i}(t)+(\bar{\phi}_{i,x}\left(  x\right)  )^{\top
}f_{u}^{i}(t)\right)  dt+(\check{X}_{t_{i+1}}^{t_{i},x})^{\top}V_{x}%
^{i+1}(\bar{X}_{t_{i+1}}^{t_{i},x})\right]  ,\text{ }x\in \mathbb{R}%
^{n}\text{,} \label{V'x}%
\end{equation}
where%
\[
\left \{
\begin{array}
[c]{l}%
d\check{X}_{t}^{t_{i},x}=\left[  b_{x}^{i}(t)\check{X}_{t}^{t_{i},x}+b_{u}%
^{i}(t)\bar{\phi}_{i,x}(x)\right]  dt\\
\text{ \  \  \  \  \  \  \  \  \  \  \  \  \  \  \  \  \  \  \  \  \ }+\left[  \sigma_{x}%
^{i}(t)\check{X}_{t}^{t_{i},x}+\sigma_{u}^{i}(t)\bar{\phi}_{i,x}(x)\right]
dW_{t},\\
\check{X}_{t_{i}}^{t_{i},x}=I,\text{\  \  \ }t\in \lbrack t_{i},t_{i+1}].
\end{array}
\right.
\]
Applying It\^{o}'s formula to\ $(\check{X}_{t}^{t_{i},x})^{\top}\bar{P}%
_{t}^{t_{i},x}$ on $\left[  t_{i},t_{i+1}\right]  $, we have%
\[
\mathbb{E}\left[  (\check{X}_{t_{i+1}}^{t_{i},x})^{\top}\bar{P}_{t_{i+1}%
}^{t_{i},x}-I\bar{P}_{t_{i}}^{t_{i},x}\right]  =\mathbb{E}\left[  \int_{t_{i}%
}^{t_{i+1}}\left[  (\bar{\phi}_{i,x}(x))^{\top}\left(  (b_{u}^{i}(t))^{\top
}\bar{P}_{t}^{t_{i},x}+(\sigma_{u}^{i}(t))^{\top}\bar{Q}_{t}^{t_{i},x}\right)
-(\check{X}_{t}^{t_{i},x})^{\top}f_{x}^{i}(t)\right]  dt\right]  ,
\]
which implies%
\begin{equation}
\bar{P}_{t_{i}}^{t_{i},x}=\mathbb{E}\left[  (\check{X}_{t_{i+1}}^{t_{i}%
,x})^{\top}\bar{P}_{t_{i+1}}^{t_{i},x}+\int_{t_{i}}^{t_{i+1}}\left[
(\check{X}_{t}^{t_{i},x})^{\top}f_{x}^{i}(t)-(\bar{\phi}_{i,x}\left(
x\right)  )^{\top}\left(  (b_{u}^{i}(t))^{\top}\bar{P}_{t}^{t_{i},x}%
+(\sigma_{u}^{i}(t))^{\top}\bar{Q}_{t}^{t_{i},x}\right)  \right]  dt\right]  .
\label{P_N-1}%
\end{equation}
Notice that
\begin{equation}
V_{x}^{i+1}(\bar{X}_{t_{i+1}}^{t_{i},x})=\bar{P}_{t_{i+1}}^{t_{i+1},\bar
{X}_{t_{i+1}}^{t_{i},x}}=\bar{P}_{t_{i+1}}^{t_{i},x}. \label{9}%
\end{equation}
Combining $\left(  \ref{V'x}\right)  $, $\left(  \ref{P_N-1}\right)  $ and
$\left(  \ref{9}\right)  $, the proof is complete.
\end{proof}
\end{lemma}

\begin{proof}
[Proof of Theorem \ref{theorem1}]For simplicity of presentation, in the
following of this proof we only consider the case $n=m=d=1$. This method is
still applicable to the multi-dimensional case. For convenience, we use the
following notations:%
\[%
\begin{array}
[c]{l}%
\displaystyle V_{x}^{i}\left(  t,\lambda \right)  =V_{x}^{i}\left(  \bar{X}%
_{t}+\lambda(X_{t}^{i-1,\varepsilon}-\bar{X}_{t})\right)  ,\\
\displaystyle f_{x}^{i}\left(  t,\lambda \right)  =f_{x}\left(  t,\bar{X}%
_{t}+\lambda(X_{t}^{i,\varepsilon}-\bar{X}_{t}),\bar{\phi}_{i}(\bar{X}_{t_{i}%
})+\lambda \left(  \phi_{i}^{\varepsilon}(\bar{X}_{t_{i}})-\bar{\phi}_{i}%
(\bar{X}_{t_{i}})\right)  \right)  ,\\
\displaystyle f_{u}^{i}\left(  t,\lambda \right)  =f_{u}\left(  t,\bar{X}%
_{t}+\lambda(X_{t}^{i,\varepsilon}-\bar{X}_{t}),\bar{\phi}_{i}(\bar{X}_{t_{i}%
})+\lambda \left(  \phi_{i}^{\varepsilon}(\bar{X}_{t_{i}})-\bar{\phi}_{i}%
(\bar{X}_{t_{i}})\right)  \right)  .
\end{array}
\]

First, we consider $u_{\cdot}^{N-1,\varepsilon}$\ on $\left[  t_{N-1}%
,T\right]  $. Since $V^{N}(x)=h(x)$, $x\in \mathbb{R}$, by Taylor's expansion,
we have%
\begin{align}
&  J(u^{N-1,\varepsilon})-J(\bar{u})\nonumber \\
&  =\mathbb{E}\left[  \int_{t_{N-1}}^{T}\left[  f\left(  t,X_{t}%
^{N-1,\varepsilon},\phi_{N-1}^{\varepsilon}(\bar{X}_{t_{N-1}})\right)
-f\left(  t,\bar{X}_{t},\bar{\phi}_{N-1}(\bar{X}_{t_{N-1}})\right)  \right]
dt\right] \nonumber \\
&  \text{ \  \ }+\mathbb{E}\left[  V^{N}(X_{T}^{N-1,\varepsilon})-V^{N}(\bar
{X}_{T})\right] \nonumber \\
&  =\mathbb{E}\left[  \int_{t_{N-1}}^{T}\int_{0}^{1}f_{x}^{N-1}\left(
t,\lambda \right)  (X_{t}^{N-1,\varepsilon}-\bar{X}_{t})d\lambda dt\right.
\label{5}\\
&  \text{ \  \ }+\left.  \int_{t_{N-1}}^{T}\int_{0}^{1}f_{u}^{N-1}\left(
t,\lambda \right)  \left[  \phi_{N-1}^{\varepsilon}(\bar{X}_{t_{N-1}}%
)-\bar{\phi}_{N-1}(\bar{X}_{t_{N-1}})\right]  d\lambda dt\right] \nonumber \\
&  \text{ \  \ }+\mathbb{E}\left[  \int_{0}^{1}V_{x}^{N}\left(  T,\lambda
\right)  (X_{T}^{N-1,\varepsilon}-\bar{X}_{T})d\lambda \right]  .\nonumber
\end{align}
Applying It\^{o}'s formula to $\bar{P}_{t}\hat{X}_{t}^{N-1}$ on $\left[
t_{N-1},T\right]  $ and noting that $\bar{P}_{T}=V_{x}^{N}(\bar{X}_{T})$ and
$\hat{X}_{t_{N-1}}^{N-1}=0$, we have%
\begin{equation}
\mathbb{E}\left[  V_{x}^{N}(\bar{X}_{T})\hat{X}_{T}^{N-1}\right]
=\int_{t_{N-1}}^{T}\mathbb{E}\left[  \left(  \bar{P}_{t}b_{u}^{N-1}\left(
t\right)  +\bar{Q}_{t}\sigma_{u}^{N-1}(t)\right)  \delta \phi_{N-1}(\bar
{X}_{t_{N-1}})-f_{x}^{N-1}(t)\hat{X}_{t}^{N-1}\right]  dt. \label{6}%
\end{equation}
Since $\bar{u}_{\cdot}$ is the discrete optimal control, by Lemma
\ref{lemma1}, $\left(  \ref{5}\right)  $ and $\left(  \ref{6}\right)  $, for
any $\phi_{N-1}(\cdot)\in C_{b}^{1}\left(  \mathbb{R};U\right)  $, we obtain
\begin{align}
&  0\leq \lim_{\varepsilon \downarrow0}\frac{J(u^{N-1,\varepsilon})-J(\bar{u}%
)}{\varepsilon}\nonumber \\
&  =\mathbb{E}\left[  \int_{t_{N-1}}^{T}\left[  f_{x}^{N-1}(t)\hat{X}%
_{t}^{N-1}+f_{u}^{N-1}(t)\delta \phi_{N-1}(\bar{X}_{t_{N-1}})\right]
dt+V_{x}^{N}(\bar{X}_{T})\hat{X}_{T}^{N-1}\right] \label{7}\\
&  =\mathbb{E}\left[  \int_{t_{N-1}}^{T}H_{u}\left(  t,\bar{X}_{t},\bar{P}%
_{t},\bar{Q}_{t},\bar{\phi}_{N-1}(\bar{X}_{t_{N-1}})\right)  \delta \phi
_{N-1}(\bar{X}_{t_{N-1}})dt\right]  .\nonumber
\end{align}
Furthermore, if $\bar{\phi}_{N-1}(x)$ is an interior point of $U$, for
$x\in \mathbb{R}$, then $\delta \phi_{N-1}(x)\ $can\ be positive or negative,
which implies
\begin{equation}
\mathbb{E}\left[  \left.  \int_{t_{N-1}}^{T}H_{u}\left(  t,\bar{X}_{t},\bar
{P}_{t},\bar{Q}_{t},\bar{\phi}_{N-1}(x)\right)  dt\right \vert \bar{X}%
_{t_{N-1}}=x\right]  =0,\ P_{\bar{X}_{t_{N-1}}}\text{-a.s.} \label{8}%
\end{equation}
$\  \  \  \  \  \ $Second, based on the preceding discussion, we consider
$u_{\cdot}^{N-2,\varepsilon}$ on $\left[  t_{N-2},T\right]  $. By the
definition of $V^{N-1}\left(  \cdot \right)  \ $and Taylor's expansion, we have%
\begin{align*}
&  J(u^{N-2,\varepsilon})-J(\bar{u})\\
&  =\mathbb{E}\left[  \int_{t_{N-2}}^{t_{N-1}}\left[  f\left(  t,X_{t}%
^{N-2,\varepsilon},\phi_{N-2}^{\varepsilon}(\bar{X}_{t_{N-2}})\right)
-f\left(  t,\bar{X}_{t},\bar{\phi}_{N-2}(\bar{X}_{t_{N-2}})\right)  \right]
dt\right] \\
&  \text{ \  \ }+\mathbb{E}\left[  V^{N-1}(X_{t_{N-1}}^{N-2,\varepsilon
})-V^{N-1}(\bar{X}_{t_{N-1}})\right] \\
&  =\mathbb{E}\left[  \int_{t_{N-2}}^{t_{N-1}}\int_{0}^{1}f_{x}^{N-2}\left(
t,\lambda \right)  (X_{t}^{N-2,\varepsilon}-\bar{X}_{t})d\lambda dt\right. \\
&  \text{ \  \ }+\left.  \int_{t_{N-2}}^{t_{N-1}}\int_{0}^{1}f_{u}^{N-2}\left(
t,\lambda \right)  \left[  \phi_{N-2}^{\varepsilon}(\bar{X}_{t_{N-2}}%
)-\bar{\phi}_{N-2}(\bar{X}_{t_{N-2}})\right]  d\lambda dt\right] \\
&  \text{ \  \ }+\mathbb{E}\left[  \int_{0}^{1}V_{x}^{N-1}\left(
t_{N-1},\lambda \right)  (X_{t_{N-1}}^{N-2,\varepsilon}-\bar{X}_{t_{N-1}%
})d\lambda \right]  .
\end{align*}
By Lemma \ref{lemma1}, it yields that%
\begin{align}
&  \text{ }0\leq \lim_{\varepsilon \downarrow0}\frac{J(u^{N-2,\varepsilon
})-J(\bar{u})}{\varepsilon}\nonumber \\
&  =\mathbb{E}\left[  V_{x}^{N-1}(\bar{X}_{t_{N-1}})\hat{X}_{t_{N-1}}%
^{N-2}+\int_{t_{N-2}}^{t_{N-1}}\left[  f_{x}^{N-2}(t)\hat{X}_{t}^{N-2}%
+f_{u}^{N-2}(t)\delta \phi_{N-2}(\bar{X}_{t_{N-2}})\right]  dt\right]  .
\label{11}%
\end{align}
From $\left(  \ref{8}\right)  $, we know
\[
\mathbb{E}\left[  \int_{t_{N-1}}^{T}H_{u}\left(  t,\bar{X}_{t}^{t_{N-1}%
,x},\bar{P}_{t}^{t_{N-1},x},\bar{Q}_{t}^{t_{N-1},x},\bar{\phi}_{N-1}%
(x)\right)  dt\left(  \phi_{N-1}(x)-\bar{\phi}_{N-1}(x)\right)  \right]
\geq0,
\]
for any $\phi_{N-1}(\cdot)\in C_{b}^{1}(\mathbb{R};U)$. If $\bar{\phi}%
_{N-1}(x)$ is an interior point of $U$,
\[
\mathbb{E}\left[  \int_{t_{N-1}}^{T}H_{u}\left(  t,\bar{X}_{t}^{t_{N-1}%
,x},\bar{P}_{t}^{t_{N-1},x},\bar{Q}_{t}^{t_{N-1},x},\bar{\phi}_{N-1}%
(x)\right)  dt\right]  =0,
\]
otherwise if $\bar{\phi}_{N-1}(x)$ is a boundary point of $U$, $x$ is an
extreme point of $\bar{\phi}_{N-1}$ and $\bar{\phi}_{N-1,x}(x)=0$. Thus
\begin{equation}
\mathbb{E}\left[  \int_{t_{N-1}}^{T}H_{u}\left(  t,\bar{X}_{t}^{t_{N-1}%
,x},\bar{P}_{t}^{t_{N-1},x},\bar{Q}_{t}^{t_{N-1},x},\bar{\phi}_{N-1}%
(x)\right)  \bar{\phi}_{N-1,x}(x)dt\right]  =0. \label{13}%
\end{equation}
Seeing that $V_{x}^{N}(\bar{X}_{T})=\bar{P}_{T}$, by Lemma
\ref{relation lemma}, we derive $V_{x}^{N-1}(\bar{X}_{t_{N-1}})=\bar
{P}_{t_{N-1}}$. Applying It\^{o}'s formula to $\bar{P}_{t}\hat{X}_{t}^{N-2}$
on $\left[  t_{N-2},t_{N-1}\right]  $\ and noting that $\hat{X}_{t_{N-2}%
}^{N-2}=0$, we have%
\begin{equation}
\mathbb{E}\left[  V_{x}^{N-1}(\bar{X}_{t_{N-1}})\hat{X}_{t_{N-1}}%
^{N-2}\right]  =\int_{t_{N-2}}^{t_{N-1}}\mathbb{E}\left[  \left(  \bar{P}%
_{t}b_{u}^{N-2}(t)+\bar{Q}_{t}\sigma_{u}^{N-2}(t)\right)  \delta \phi
_{N-2}(\bar{X}_{t_{N-2}})-f_{x}^{N-2}(t)\hat{X}_{t}^{N-2}\right]  dt.
\label{16}%
\end{equation}
Combining $\left(  \ref{11}\right)  $ and $\left(  \ref{16}\right)  $, we
obtain
\[
\mathbb{E}\left[  \int_{t_{N-2}}^{t_{N-1}}H_{u}\left(  t,\bar{X}_{t},\bar
{P}_{t},\bar{Q}_{t},\bar{\phi}_{N-2}(\bar{X}_{t_{N-2}})\right)  \delta
\phi_{N-2}(\bar{X}_{t_{N-2}})dt\right]  \geq0,\text{ \ }\forall \phi
_{N-2}(\cdot)\in C_{b}^{1}\left(  \mathbb{R};U\right)  .
\]
Furthermore, if $\bar{\phi}_{N-2}(x)$ is an interior point of $U$, for
$x\in \mathbb{R}$, then
\[
\mathbb{E}\left[  \left.  \int_{t_{N-2}}^{t_{N-1}}H_{u}\left(  t,\bar{X}%
_{t},\bar{P}_{t},\bar{Q}_{t},\bar{\phi}_{N-2}(x)\right)  dt\right \vert \bar
{X}_{t_{N-2}}=x\right]  =0,\text{ \ }P_{\bar{X}_{t_{N-2}}}\text{-a.s.}%
\]
Similarly, consider the above process on $\left[  t_{N-3},T\right]  ,\left[
t_{N-4},T\right]  ,\ldots,\left[  t_{0},T\right]  $, and then our conclusion follows.
\end{proof}

To sum up, for $i=N-1,\ldots,1,0$, the discrete optimal control $\bar{u}%
_{t}=\bar{\phi}_{i}(\bar{X}_{t_{i}})$ on $t\in \lbrack t_{i},t_{i+1}]$ can be
determined by the following discrete Hamiltonian system:
\begin{align}
&  \bar{X}_{t}=\bar{X}_{t_{i}}+\int_{t_{i}}^{t}b\left(  s,\bar{X}_{s}%
,\bar{\phi}_{i}(\bar{X}_{t_{i}})\right)  ds+\int_{t_{i}}^{t}\sigma \left(
s,\bar{X}_{s},\bar{\phi}_{i}(\bar{X}_{t_{i}})\right)  dW_{s},\label{10.1}\\
&  \bar{P}_{t}=\bar{P}_{t_{i+1}}+\int_{t}^{t_{i+1}}H_{x}\left(  s,\bar{X}%
_{s},\bar{P}_{s},\bar{Q}_{s},\bar{\phi}_{i}(\bar{X}_{t_{i}})\right)
ds-\int_{t}^{t_{i+1}}\bar{Q}_{s}dW_{s},\label{10.2}\\
&  \mathbb{E}\left[  \left \langle \int_{t_{i}}^{t_{i+1}}H_{u}\left(  t,\bar
{X}_{t},\bar{P}_{t},\bar{Q}_{t},\bar{\phi}_{i}(\bar{X}_{t_{i}})\right)
dt,\phi_{i}(\bar{X}_{t_{i}})-\bar{\phi}_{i}(\bar{X}_{t_{i}})\right \rangle
\right]  \geq0,\  \forall \phi_{i}( \cdot) \in C_{b}^{1}(\mathbb{R}^{n};U),
\label{10.3}%
\end{align}
where $\bar{X}_{0}=x_{0}$ and $\bar{P}_{T}=h_{x}(\bar{X}_{T})$.

\subsection{Numerical approach for SOCPs}

The discrete SMP provides a necessary condition for solving the discrete
optimal control. The primary challenge in applying the discrete SMP to solve
the discrete optimal control $\bar{u}_{t}$ is to obtain $H_{u}(t,\bar{X}%
_{t},\bar{P}_{t},\bar{Q}_{t},\bar{u}_{t})$, which is described by the solution
$(\bar{X}_{t},\bar{P}_{t},\bar{Q}_{t})$ of the FBSDEs $\left(  \ref{10.1}%
\right)  -\left(  \ref{10.2}\right)  $. Let $P_{i}^{N}\left(  x\right)  $,
$Q_{i}^{N}\left(  x\right)  $ and $\phi_{i}^{N}\left(  x\right)  $ be the
discrete approximations of $\bar{P}\left(  t_{i},x\right)  $, $\bar{Q}\left(
t_{i},x\right)  $ and $\bar{\phi}_{i}\left(  x\right)  $ respectively. Based
on the approximation of the discrete SMP $\left(  \ref{10.3}\right)  $ and the
FBSDEs $\left(  \ref{10.1}\right)  -\left(  \ref{10.2}\right)  $, we propose
the following numerical scheme for the discrete Hamiltonian system $\left(
\ref{10.1}\right)  -\left(  \ref{10.3}\right)  $.
%

\begin{sch}
\label{Scheme1}Assume that $X_{0}^{N}$ and $P_{N}^{N}$ are known. For
$i=N-1,\ldots,1,0$, solve $P_{i}^{N}(x),Q_{i}^{N}(x)$ and $\phi_{i}^{N}(x)$
with $x\in \mathbb{R}^{n}$ by%
\begin{align}
&  X_{i+1}^{N}=x+b\left(  t_{i},x,\phi_{i}^{N}(x)\right)  \Delta
t+\sigma \left(  t_{i},x,\phi_{i}^{N}(x)\right)  \Delta W_{t_{i+1}}%
,\label{X^N}\\
&  Q_{i}^{N}(x)=\mathbb{E}_{t_{i}}^{x}\left[  P_{i+1}^{N}\Delta W_{t_{i+1}%
}\right]  /\Delta t,\label{Q^N}\\
&  P_{i}^{N}(x)=\mathbb{E}_{t_{i}}^{x}\left[  P_{i+1}^{N}\right]
+H_{x}\left(  t_{i},x,P_{i}^{N}(x),Q_{i}^{N}(x),\phi_{i}^{N}(x)\right)  \Delta
t,\label{P^N}\\
&  \left \langle H_{u}\left(  t_{i},x,P_{i}^{N}(x),Q_{i}^{N}(x),\phi_{i}%
^{N}(x)\right)  ,\; \phi_{i}(x)-\phi_{i}^{N}(x)\right \rangle \geq
0,\text{\  \ }\forall \phi_{i}(x)\in U, \label{fai^N}%
\end{align}
where $P_{i+1}^{N}\ $is the value at space point $X_{i+1}^{N,t_{i},x}$.
\end{sch}

In the case where $\phi_{i}^{N}(x) $ is an interior point of $U$, Scheme
\ref{Scheme1} becomes Scheme \ref{Scheme2}.

\begin{sch}
\label{Scheme2}Assume that $X_{0}^{N}$ and $P_{N}^{N}$ are known. For
$i=N-1,\ldots,1,0$, solve $P_{i}^{N}(x),Q_{i}^{N}(x)$ and $\phi_{i}^{N}(x)$
with $x\in \mathbb{R}^{n}$ by%
\begin{align}
&  X_{i+1}^{N}=x+b\left(  t_{i},x,\phi_{i}^{N}(x)\right)  \Delta
t+\sigma \left(  t_{i},x,\phi_{i}^{N}(x)\right)  \Delta W_{t_{i+1}},\\
&  Q_{i}^{N}(x)=\mathbb{E}_{t_{i}}^{x}\left[  P_{i+1}^{N}\Delta W_{t_{i+1}%
}\right]  /\Delta t,\\
&  P_{i}^{N}(x)=\mathbb{E}_{t_{i}}^{x}\left[  P_{i+1}^{N}\right]
+H_{x}\left(  t_{i},x,P_{i}^{N}(x),Q_{i}^{N}(x),\phi_{i}^{N}(x)\right)  \Delta
t,\\
&  H_{u}\left(  t_{i},x,P_{i}^{N}(x),Q_{i}^{N}(x),\phi_{i}^{N}(x)\right)  =0,
\label{zero point}%
\end{align}
where $P_{i+1}^{N}\ $is the value at space point $X_{i+1}^{N,t_{i},x}$.
\end{sch}

Moreover, denote
\begin{equation}
u_{t}^{N}=\sum_{i=0}^{N-1}\phi_{i}^{N}(X_{t_{i}}^{N})I_{\left[  t_{i}%
,t_{i+1}\right)  }(t) , \label{u^N}%
\end{equation}
where the state process
\begin{equation}
\left \{
\begin{array}
[c]{l}%
dX_{t}^{N}=b\left(  t,X_{t}^{N},\phi_{i}^{N}(X_{t_{i}}^{N})\right)
dt+\sigma \left(  t,X_{t}^{N},\phi_{i}^{N}(X_{t_{i}}^{N})\right)  dW_{t},\\
X_{t_{i}}^{N}=X_{t_{i}}^{N},\  \ t\in \lbrack t_{i},t_{i+1}],\text{
}i=0,1,\ldots,N-1,
\end{array}
\right.  \label{27}%
\end{equation}
and $X_{0}^{N}=x_{0}$. Then%
\[
J\left(  u^{N}\right)  =\mathbb{E}\left[  \int_{0}^{T}f\left(  t,X_{t}%
^{N},u_{t}^{N}\right)  dt+h(X_{T}^{N})\right]  .
\]

\begin{remark}
Numerical methods for FBSDEs have been a hot topic recently (see
\cite{B1997,ET2018,GLW2005,MT2006,ZhCP2006,ZhFZh2014} and the references
therein). In this paper, we choose the Euler-type method for solving FBSDEs
proposed in \cite{ZhCP2006}\ and \cite{ZhFZh2014}. The conditional
expectations $\mathbb{E}_{t_{i}}^{x}\left[  P_{i+1}^{N}\right]  :=\mathbb{E}%
\left[  P_{i+1}^{N}|X_{t_{i}}^{N}=x\right]  $ and $\mathbb{E}_{t_{i}}%
^{x}\left[  P_{i+1}^{N}\Delta W_{t_{i+1}}\right]  :$ $=\mathbb{E}\left[
P_{i+1}^{N}\Delta W_{t_{i+1}}|X_{t_{i}}^{N}=x\right]  $ in Scheme
\ref{Scheme1} and Scheme \ref{Scheme2} are functions of Gaussian random
variables, which can be approximated by Gauss-Hermite quadrature with high accuracy.
\end{remark}

\begin{remark}
For fixed $x$, $H_{u}\left(  t_{i},x,P_{i}^{N},Q_{i}^{N},\phi_{i}%
^{N}(x)\right)  $ in $\left(  \ref{fai^N}\right)  $ and $\left(
\ref{zero point}\right)  $ is a deterministic function of the variable
$y=\phi_{i}^{N}(x)$. Many classical numerical methods can be used to solve
$\left(  \ref{fai^N}\right)  $ and $\left(  \ref{zero point}\right)  $, such
as gradient descent method, fixed-point iterative method, Newton's Method,
Bisection method and so on. We assume that $\phi_{i}^{N}(x)$ can be solved accurately.
\end{remark}

\begin{remark}
Once the control $u_{\cdot}^{N}$ is obtained, by introducing the following
equation%
\begin{equation}
Y_{t}^{N}=h\left(  X_{T}^{N}\right)  +\int_{t}^{T}f\left(  s,X_{s}^{N}%
,u_{s}^{N}\right)  ds-\int_{t}^{T}Z_{s}^{N}dW_{s}, \label{bsde}%
\end{equation}
then the cost $J\left(  u^{N}\right)  =Y_{0}^{N}$ can be obtained by solving
the FBSDEs $\left(  \ref{27}\right)  -\left(  \ref{bsde}\right)  $.
\end{remark}

\subsubsection{Summary of the discrete recursive algorithm}

To do this, we introduce the following uniform space partition $\mathcal{D}%
_{h}=\mathcal{D}_{1,h}\times \mathcal{D}_{2,h}\times \cdots \times \mathcal{D}%
_{n,h}$, where $\mathcal{D}_{j,h}$ is the partition of the one-dimensional
real axis $\mathbb{R}$
\[
\mathcal{D}_{j,h}=\left \{  \left.  x_{k}^{j}\right \vert x_{k}^{j}=kh,\text{
}k=0,\pm1,\pm2,\ldots \right \}  ,
\]
for $j=1,2,\ldots,n$ and $h$ is a suitable spatial step.

In the numerical algorithm, we employ the following iterative algorithm to
optimize control
\begin{equation}
\phi_{i}^{N,l+1}(x)=\phi_{i}^{N,l}(x)-\rho^{l}H_{u}\left(  t_{i},x,P_{i}%
^{N,l}(x),Q_{i}^{N,l}(x),\phi_{i}^{N,l}(x)\right)  ,\text{ }l=0,1,\cdots,
\label{gradient descent}%
\end{equation}
where $\rho^{l}$ is the step-size for the iteration. Now we summarize our
discrete recursive algorithm. \begin{algorithm}
\caption{Framework of the discrete recursive method}
\begin{algorithmic}[1]\label{alg1}
\STATE Set $P_{N}^{N}(x)=h_x(x)$, $x \in \mathcal{D}_h$ and the error tolerance $\varepsilon$.
\FOR{$i=N-1 \rightarrow 0$}
\FOR{each\;$x\in \mathcal{D}_h$}
\STATE Choose $\phi_{i}^{N,0}(x)\in U$ and set $l=0$.
\REPEAT
\STATE Solve $(  P_{i}^{N,l}(x), Q_{i}^{N,l}(x))  $ by $\left(
\ref{Q^N}\right)-\left(  \ref{P^N}\right)$.
\STATE {Update $\phi_{i}^{N,l+1}(x)$ by $\eqref{gradient descent}$.
Set $l=l+1$.}
\UNTIL{$|\phi_{i}^{N,l+1}(x)-\phi_{i}^{N,l}(x)| \leq \varepsilon$.}
\ENDFOR
\ENDFOR
\STATE Compute $u_{t}^{N}$ by $\eqref{u^N}$.
\end{algorithmic}
\end{algorithm}

Algorithm \ref{alg1} presents the procedure for our discrete recursive method.
We run the algorithm in a backward manner to obtain the values $\{ \phi
_{i}^{N}(x) \}_{i=0}^{N-1}$, $x\in \mathcal{D}_{h}$, which are the control
values in time-space mesh. Then we can compute the control value $u_{t}^{N}$
based on grid point interpolation.

\section{Convergence analysis}

We will give the convergence results of the discrete recursive method in this
section. In the following, $C$ represents a generic constant which does not
depend on the time partition and may be different from line to line. We now
give an estimate for the state process $X_{t}^{N}$.

\begin{lemma}
\label{SDE_Estimate}Suppose $\left(  A1\right)  \ $holds. We also assume $\{
\phi_{i}^{N}\}_{i=1}^{N-1}\in C_{b}^{1}$, and there exists a positive constant
$L$, not depending on $N$, such that $\sup_{i}\left \vert \phi_{i}^{N}\left(
0\right)  \right \vert \leq L$. Then for $m\geq2$,
\begin{equation}
\mathbb{E}\left[  \sup_{0\leq s\leq T}\left \vert X_{s}^{N}\right \vert
^{m}\right]  \leq C\left(  1+\left \vert x_{0}\right \vert ^{m}\right)  .
\label{sde_esti}%
\end{equation}

\end{lemma}

\begin{proof}
Rewrite the state equation $\left(  \ref{27}\right)  $\ as follows:%
\[
X_{s}^{N}=x_{0}+\int_{0}^{s}\tilde{b}\left(  r,X_{r}^{N}\right)  dr+\int
_{0}^{s}\tilde{\sigma}\left(  r,X_{r}^{N}\right)  dW_{r},
\]
where\ for $r\in \left[  t_{i},t_{i+1}\right]  $ $\left(  i=0,\ldots
,N-1\right)  ,$%
\[%
\begin{array}
[c]{ll}%
\tilde{b}\left(  r,x\right)  =b\left(  r,x,\phi_{i}^{N}(X_{t_{i}}^{N})\right)
, & \tilde{\sigma}\left(  r,x\right)  =\sigma \left(  r,x,\phi_{i}^{N}%
(X_{t_{i}}^{N})\right)  .
\end{array}
\]
For $m\geq2$, by the standard estimate of SDE, one can derive that
\begin{align*}
\mathbb{E}\left[  \sup_{0\leq s\leq t}\left \vert X_{s}^{N}\right \vert
^{m}\right]   &  \leq C\left \{  \left \vert x_{0}\right \vert ^{m}%
+\mathbb{E}\left[  \int_{0}^{t}|\tilde{b}\left(  s,X_{s}^{N}\right)
|^{m}ds\right]  +\mathbb{E}\left[  \int_{0}^{t}\left \vert \tilde{\sigma
}\left(  s,X_{s}^{N}\right)  \right \vert ^{m}ds\right]  \right \} \\
&  \leq C\left \{  \left \vert x_{0}\right \vert ^{m}+\int_{0}^{t}\mathbb{E}%
\left[  |\tilde{b}\left(  s,0\right)  |^{m}+\left \vert \tilde{\sigma}\left(
s,0\right)  \right \vert ^{m}+\left \vert X_{s}^{N}\right \vert ^{m}\right]
ds\right \}  .
\end{align*}
Notice that for $s\in \left[  t_{i},t_{i+1}\right]  ,$
\begin{align*}
|\tilde{b}\left(  s,0\right)  |  &  \leq C\left(  1+\left \vert \phi_{i}%
^{N}(X_{t_{i}}^{N})\right \vert \right)  \leq C\left(  1+\left \vert X_{t_{i}%
}^{N}\right \vert \right)  ,\\
\left \vert \tilde{\sigma}\left(  s,0\right)  \right \vert  &  \leq C\left(
1+\left \vert \phi_{i}^{N}(X_{t_{i}}^{N})\right \vert \right)  \leq C\left(
1+\left \vert X_{t_{i}}^{N}\right \vert \right)  .
\end{align*}
Hence
\[
\mathbb{E}\left[  \sup_{0\leq s\leq t}\left \vert X_{s}^{N}\right \vert
^{m}\right]  \leq C\left(  1+\left \vert x_{0}\right \vert ^{m}\right)
+C\int_{0}^{t}\mathbb{E}\left[  \sup_{0\leq r\leq s}\left \vert X_{r}%
^{N}\right \vert ^{m}\right]  ds,
\]
where $C\ $is a positive constant not depending on $N$. By the Gronwall
inequality, the required result $\left(  \ref{sde_esti}\right)  $\ follows.
\end{proof}

\begin{remark}
The above conclusion also holds for state processes $X_{t}^{\ast}$ and
$\bar{X}_{t}$.
\end{remark}

We need the following assumption:

\begin{description}
\item[$\left(  A2\right)  $] There exists a constant $c_{0}>0$, such that for
each $\left(  t_{i},x\right)  \in \left[  0,T\right]  \times \mathbb{R}^{n}$,
\[%
\begin{array}
[c]{r}%
\left \langle H_{u}\left(  t_{i},x,P_{t_{i}}^{N},Q_{t_{i}}^{N},\phi_{i}%
^{N}(x)\right)  -H_{u}\left(  t_{i},x,\bar{P}_{t_{i}},\bar{Q}_{t_{i}}%
,\bar{\phi}_{i}(x)\right)  ,\text{ }\phi_{i}^{N}(x)-\bar{\phi}_{i}%
(x)\right \rangle \geq c_{0}\left \vert \phi_{i}^{N} (x)-\bar{\phi}%
_{i}(x)\right \vert ^{2}\text{, }%
\end{array}
\]
where $(\bar{P}_{\cdot},\bar{Q}_{\cdot})$ and $(P_{\cdot}^{N},Q_{\cdot}^{N})$
are the adjoint processes with respect to $\bar{u}_{\cdot}$ and $u_{\cdot}%
^{N}$, respectively.
\end{description}

\begin{remark}
For fixed $\left(  t_{i},x\right)  \in \left[  0,T\right]  \times \mathbb{R}%
^{n}$, $H_{u}(t_{i},x,\bar{P}_{t_{i}},\bar{Q}_{t_{i}},\bar{\phi}_{i}(x))$ is a
deterministic function of $\bar{\phi}_{i}(x)$. The above assumption means that
$\bar{H}(u_{t_{i}}^{x}):=H\left(  t_{i},x,P_{t_{i}}^{u},Q_{t_{i}}^{u}%
,u_{t_{i}}^{x}\right)  $ is uniformly monotone around $u_{t_{i}}^{x}=\bar
{\phi}_{i}(x)$, that is, when $\bar{H}(\phi_{i}^{N}(x))$ and $\bar{H}%
(\bar{\phi}_{i}(x))$ are close, $\phi_{i}^{N}(x)$\ and $\bar{\phi}_{i}(x)$ are
also close. In particular, if $U$ is an open set, the above assumption is also
true for $\left \langle \bar{H}(\phi_{i}^{N}(x))-\bar{H}(\bar{\phi}%
_{i}(x)),\phi_{i}^{N}(x)-\bar{\phi}_{i}(x)\right \rangle \leq-c_{0}\left \vert
\phi_{i}^{N}(x)-\bar{\phi}_{i}(x)\right \vert ^{2}$.
\end{remark}

Now we state our main convergence result.

\begin{theorem}
\label{theorem}Suppose $\left(  A1\right)  -\left(  A2\right)  $ hold. We also
assume $b,\sigma \in C_{b}^{2,5,5}$, $f\in C_{b}^{2,5,5}$, $\phi^{\ast}\in
C_{b}^{2,4+\alpha}$, $\{ \bar{\phi}_{i}\}_{i=1}^{N-1},\{ \phi_{i}^{N}%
\}_{i=1}^{N-1},\{P_{i}^{N}\}_{i=1}^{N-1}\in C_{b}^{4}$ and $h\in
C_{b}^{5+\alpha}$, $\alpha>0$, and there exists a positive constant $L$, not
depending on $N$, such that $\sup_{i}\left(  \left \vert \phi_{i}^{N}\left(
0\right)  \right \vert +\left \vert \bar{\phi}_{i}\left(  0\right)  \right \vert
\right)  \leq L$.\ Then for sufficiently small time step $\Delta t$,
\[
\left \vert J\left(  u^{\ast}\right)  -J\left(  u^{N}\right)  \right \vert \leq
C\Delta t.
\]

\end{theorem}

The proof of our convergence theorem will be divided into two parts: discrete
approximation $\vert J\left(  u^{\ast}\right)  -J\left(  \bar{u}\right)
\vert$ and recursive approximation $|J\left(  \bar{u}\right)  -J(u^{N})|$.

\subsection{Discrete approximation}

In this subsection, we give the estimate of the discrete approximation error
$\left \vert J\left(  u^{\ast}\right)  -J\left(  \bar{u}\right)  \right \vert $.

\begin{theorem}
\label{theorem2}Suppose $\left(  A1\right)  \ $holds. We also assume that
$b,\sigma \in C_{b}^{2,4,4},f\in C_{b}^{2,4+\alpha,4+\alpha},\phi^{\ast}\in
C_{b}^{2,4+\alpha}$ and $h\in C_{b}^{4+\alpha}$, $\alpha>0$. Then for
sufficiently small $\Delta t$,
\begin{equation}
\left \vert J\left(  u^{\ast}\right)  -J\left(  \bar{u}\right)  \right \vert
\leq C\Delta t. \label{Discrete error}%
\end{equation}

\begin{proof}
For simplicity of presentation, in the following of this proof we only
consider the case $n=1$. Conclusions still hold for the case $n>1$. To begin
with, we define
\begin{equation}
\tilde{u}_{t}=\sum_{i=0}^{N-1}\phi^{\ast}(t_{i},\tilde{X}_{t_{i}})I_{\left[
t_{i},t_{i+1}\right)  }(t),
\end{equation}
where
\[
\left \{
\begin{array}
[c]{l}%
d\tilde{X}_{t}=b(t,\tilde{X}_{t},\phi^{\ast}(t_{i},\tilde{X}_{t_{i}}%
))+\sigma(t,\tilde{X}_{t},\phi^{\ast}(t_{i},\tilde{X}_{t_{i}}))dW_{t},\\
\tilde{X}_{t_{i}}=\tilde{X}_{t_{i}},\text{ \ }t\in \left[  t_{i},t_{i+1}%
\right]  ,\text{ }i=0,1,\ldots,N-1,
\end{array}
\right.
\]
and $\tilde{X}_{0}=x_{0}.$ Since $\tilde{u}_{\cdot}\in \mathcal{U}^{N}[0,T]$,
we have $J\left(  u^{\ast}\right)  \leq J\left(  \bar{u}\right)  \leq J\left(
\tilde{u}\right)  $. So to prove $\left(  \ref{Discrete error}\right)  $, it
suffices to prove
\begin{equation}
\left \vert J\left(  u^{\ast}\right)  -J\left(  \tilde{u}\right)  \right \vert
\leq C\Delta t.\label{20}%
\end{equation}
Let $J\left(  u^{\ast}\right)  -J\left(  \tilde{u}\right)  =J_{1}+J_{2}$,
where
\[%
\begin{array}
[c]{l}%
\displaystyle J_{1}=\int_{0}^{T}\mathbb{E}\left[  f\left(  t,X_{t}^{\ast}%
,\phi^{\ast}(t,X_{t}^{\ast})\right)  -f(t,\tilde{X}_{t},\phi^{\ast}%
(t,\tilde{X}_{t}))\right]  dt+\mathbb{E}\left[  h\left(  X_{T}^{\ast}\right)
-h(\tilde{X}_{T})\right]  ,\\
\displaystyle J_{2}=\int_{0}^{T}\sum_{i=0}^{N-1}\mathbb{E}\left[
f(t,\tilde{X}_{t},\phi^{\ast}(t,\tilde{X}_{t}))-f(t,\tilde{X}_{t},\phi^{\ast
}(t_{i},\tilde{X}_{t_{i}}))\right]  I_{\left[  t_{i},t_{i+1}\right)  }(t)dt.
\end{array}
\]
Note that
\[
X_{t}^{\ast}=x_{0}+\int_{0}^{t}b^{\ast}\left(  s,X_{s}^{\ast}\right)
ds+\int_{0}^{t}\sigma^{\ast}\left(  s,X_{s}^{\ast}\right)  dW_{s},\text{ }%
t\in \left[  0,T\right]  ,
\]
where $b^{\ast}\left(  s,X_{s}^{\ast}\right)  :=b(s,X_{s}^{\ast},\phi^{\ast
}(s,X_{s}^{\ast}))$ and $\sigma^{\ast}\left(  s,X_{s}^{\ast}\right)
:=\sigma(s,X_{s}^{\ast},\phi^{\ast}(s,X_{s}^{\ast}))$. Since\ $b^{\ast}%
,\sigma^{\ast}\in C_{b}^{2,4}$, $f\in C_{b}^{2,4+\alpha,4+\alpha}$,
$\phi^{\ast}\in C_{b}^{2,4+\alpha}$ and $h\in C_{b}^{4+\alpha}$, $\alpha>0$,
by Remark \ref{remark_F-K}, we then have%
\begin{equation}%
\begin{array}
[c]{ll}%
v\left(  s,x;t\right)  =\mathbb{E}\left[  f\left(  t,X_{t}^{\ast,s,x}%
,\phi^{\ast}(t,X_{t}^{\ast,s,x})\right)  \right]  , & \mu \left(  t,x\right)
=\mathbb{E}\left[  h\left(  X_{T}^{\ast,t,x}\right)  \right]  ,
\end{array}
\label{28}%
\end{equation}
where $\mu \left(  \cdot,\cdot \right)  ,v\left(  \cdot,\cdot;t\right)  \in
C_{b}^{2,4}$ are the solution of $\left(  \ref{SDE-PDE}\right)  $ with the
terminal $\mu \left(  T,x\right)  =h(x)$ and $v\left(  t,x;t\right)  =f\left(
t,x,\phi^{\ast}\left(  t,x\right)  \right)  $, respectively. By applying
It\^{o}'s formula\ to $\mu \left(  T,X_{T}^{\ast}\right)  $ and $v\left(
t,X_{t}^{\ast};t\right)  $, from $\left(  \ref{SDE-PDE}\right)  $, we obtain
\begin{equation}%
\begin{array}
[c]{l}%
\displaystyle \mathbb{E}\left[  v\left(  t,X_{t}^{\ast};t\right)  \right]
=\mathbb{E}\left[  v\left(  0,x_{0};t\right)  \right]  +\int_{0}^{t}%
\mathbb{E}\left[  \mathcal{L}v\left(  s,X_{s}^{\ast};t\right)  \right]
ds=\mathbb{E}\left[  v\left(  0,x_{0};t\right)  \right]  ,\\
\displaystyle \mathbb{E}\left[  \mu \left(  T,X_{T}^{\ast}\right)  \right]
=\mathbb{E}\left[  \mu \left(  0,x_{0}\right)  \right]  +\int_{0}^{T}%
\mathbb{E}\left[  \mathcal{L}\mu \left(  s,X_{s}^{\ast}\right)  \right]
ds=\mathbb{E}\left[  \mu \left(  0,x_{0}\right)  \right]  .
\end{array}
\label{26}%
\end{equation}
On the one hand, combining $\left(  \ref{28}\right)  -\left(  \ref{26}\right)
$, we have
\begin{align}
\left \vert J_{1}\right \vert  &  \leq \int_{0}^{T}\left \vert \mathbb{E}\left[
v\left(  t,X_{t}^{\ast};t\right)  -v(t,\tilde{X}_{t};t)\right]  \right \vert
dt+\left \vert \mathbb{E}\left[  \mu \left(  T,X_{T}^{\ast}\right)
-\mu(T,\tilde{X}_{T})\right]  \right \vert \label{30}\\
&  \leq \int_{0}^{T}\left \vert \mathbb{E}\left[  v(t,\tilde{X}_{t};t)-v\left(
0,x_{0};t\right)  \right]  \right \vert dt+\left \vert \mathbb{E}\left[
\mu(T,\tilde{X}_{T})-\mu \left(  0,x_{0}\right)  \right]  \right \vert
.\nonumber
\end{align}
Set%
\[%
\begin{array}
[c]{rr}%
b_{i}(t,\tilde{X}_{t})=b(t,\tilde{X}_{t},\phi^{\ast}(t_{i},\tilde{X}_{t_{i}%
})), & \sigma_{i}(t,\tilde{X}_{t})=\sigma(t,\tilde{X}_{t},\phi^{\ast}%
(t_{i},\tilde{X}_{t_{i}})).
\end{array}
\]
By It\^{o}'s formula and $\left(  \ref{SDE-PDE}\right)  $, we have%
\begin{align}
\left \vert \mathbb{E}\left[  \mu(T,\tilde{X}_{T})-\mu \left(  0,x_{0}\right)
\right]  \right \vert  &  \leq \sum_{i=0}^{N-1}\left \vert \mathbb{E}\left[
\mu(t_{i+1},\tilde{X}_{t_{i+1}})-\mu(t_{i},\tilde{X}_{t_{i}})\right]
\right \vert \\
&  \leq \sum_{i=0}^{N-1}\left \vert \int_{t_{i}}^{t_{i+1}}\mathbb{E}\left[
\partial_{t}\mu(s,\tilde{X}_{s})+b_{i}(s,\tilde{X}_{s})\partial_{x}%
\mu(s,\tilde{X}_{s})\right.  \right.  \nonumber \\
&  \text{ \  \  \ }+\left.  \left.  \frac{1}{2}\sigma_{i}^{2}(s,\tilde{X}%
_{s})\partial_{xx}^{2}\mu(s,\tilde{X}_{s})-\mathcal{L}\mu(t_{i},\tilde
{X}_{t_{i}})\right]  ds\right \vert ,\text{ \  \  \  \  \  \  \  \  \  \  \  \  \  \  \  \  \ }%
\nonumber
\end{align}
which implies%
\begin{align}
\left \vert \mathbb{E}\left[  \mu(T,\tilde{X}_{T})-\mu \left(  0,x_{0}\right)
\right]  \right \vert  &  \leq \sum_{i=0}^{N-1}\left \vert \int_{t_{i}}^{t_{i+1}%
}\mathbb{E}\left[  \left(  \partial_{t}\mu(s,\tilde{X}_{s})-\partial_{t}%
\mu(t_{i},\tilde{X}_{t_{i}})\right)  \right.  \right.  \label{21}\\
&  \text{ \  \  \ }+\left(  b_{i}(s,\tilde{X}_{s})\partial_{x}\mu(s,\tilde
{X}_{s})-b_{i}(t_{i},\tilde{X}_{t_{i}})\partial_{x}\mu(t_{i},\tilde{X}_{t_{i}%
})\right)  \text{\  \  \ }\nonumber \\
&  \text{ \  \  \ }+\left.  \left.  \frac{1}{2}\left(  \sigma_{i}^{2}%
(s,\tilde{X}_{s})\partial_{xx}^{2}\mu(s,\tilde{X}_{s})-\sigma_{i}^{2}%
(t_{i},\tilde{X}_{t_{i}})\partial_{xx}^{2}\mu(t_{i},\tilde{X}_{t_{i}})\right)
\right]  ds\right \vert .\nonumber
\end{align}
Using It\^{o}'s formula again, we have%
\begin{align}
&  \left \vert \mathbb{E}\left[  \partial_{t}\mu(s,\tilde{X}_{s})-\partial
_{t}\mu(t_{i},\tilde{X}_{t_{i}})\right]  \right \vert \label{22}\\
&  \leq \int_{t_{i}}^{s}\mathbb{E}\left[  \left \vert \partial_{tt}^{2}%
\mu(r,\tilde{X}_{r})+b_{i}(r,\tilde{X}_{r})\partial_{tx}^{2}\mu(r,\tilde
{X}_{r})+\frac{1}{2}\sigma_{i}^{2}(r,\tilde{X}_{r})\partial_{txx}^{3}%
\mu(r,\tilde{X}_{r})\right \vert \right]  dr\text{
\  \  \  \  \  \  \  \  \  \  \  \  \  \  \  \  \  \ }\nonumber \\
&  \leq C\Delta t+C\int_{t_{i}}^{s}\mathbb{E}\left[  1+|\tilde{X}_{r}%
|^{2}\right]  dr.\nonumber
\end{align}
Similarly, we can obtain
\begin{align}
&  \left \vert \mathbb{E}\left[  b_{i}(s,\tilde{X}_{s})\partial_{x}\mu
(s,\tilde{X}_{s})-b_{i}(t_{i},\tilde{X}_{t_{i}})\partial_{x}\mu(t_{i}%
,\tilde{X}_{t_{i}})\right]  \right \vert \leq C\Delta t+C\int_{t_{i}}%
^{s}\mathbb{E}\left[  1+|\tilde{X}_{r}|^{3}\right]  dr,\\
&  \left \vert \mathbb{E}\left[  \sigma_{i}^{2}(s,\tilde{X}_{s})\partial
_{xx}^{2}\mu(s,\tilde{X}_{s})-\sigma_{i}^{2}(t_{i},\tilde{X}_{t_{i}}%
)\partial_{xx}^{2}\mu(t_{i},\tilde{X}_{t_{i}})\right]  \right \vert \leq
C\Delta t+C\int_{t_{i}}^{s}\mathbb{E}\left[  1+|\tilde{X}_{r}|^{4}\right]
dr.\label{25}%
\end{align}
From $\left(  \ref{21}\right)  -\left(  \ref{25}\right)  $, by Lemma
\ref{SDE_Estimate}, it follows\ that%
\begin{equation}
\left \vert \mathbb{E}\left[  \mu(T,\tilde{X}_{T})-\mu \left(  0,x_{0}\right)
\right]  \right \vert \leq C\Delta t.
\end{equation}
In the same way, we can estimate $\int_{0}^{T}|\mathbb{E}[v(t,\tilde{X}%
_{t};t)-v(0,x_{0};t)]|dt\leq C\Delta t$. Then, by $\left(  \ref{30}\right)  $,
it follows that $\left \vert J_{1}\right \vert \leq C\Delta t$. On the other
hand, seeing that
\begin{align*}
\left \vert J_{2}\right \vert  &  \leq \sum_{i=0}^{N-1}\int_{t_{i}}^{t_{i+1}%
}\left \vert \mathbb{E}\left[  f(t,\tilde{X}_{t},\phi^{\ast}(t,\tilde{X}%
_{t}))-f(t,\tilde{X}_{t},\phi^{\ast}(t_{i},\tilde{X}_{t_{i}}))\right]
\right \vert dt\text{ \  \  \  \  \  \  \  \  \  \  \  \  \  \  \  \  \ }\\
&  =\sum_{i=0}^{N-1}\int_{t_{i}}^{t_{i+1}}\left \vert \mathbb{E}\left[
f(t,\tilde{X}_{t},\phi^{\ast}(t,\tilde{X}_{t}))-f(t_{i},\tilde{X}_{t_{i}}%
,\phi^{\ast}(t_{i},\tilde{X}_{t_{i}}))\right]  \right.  \\
&  \text{ \  \  \ }-\left.  \mathbb{E}\left[  f(t,\tilde{X}_{t},\phi^{\ast
}(t_{i},\tilde{X}_{t_{i}}))-f(t_{i},\tilde{X}_{t_{i}},\phi^{\ast}(t_{i}%
,\tilde{X}_{t_{i}}))\right]  \right \vert dt.
\end{align*}
By It\^{o}'s formula, we have
\begin{align}
\left \vert J_{2}\right \vert  &  \leq \sum_{i=0}^{N-1}\int_{t_{i}}^{t_{i+1}%
}\left \{  \int_{t_{i}}^{t}\mathbb{E}\left[  \left \vert \partial_{t}%
f(s,\tilde{X}_{s},\phi^{\ast}(s,\tilde{X}_{s}))+b_{i}(s,\tilde{X}_{s}%
)\partial_{x}f(s,\tilde{X}_{s},\phi^{\ast}(s,\tilde{X}_{s}))\right.  \right.
\right.  \nonumber \\
&  \text{\  \  \  \ }+\left.  \left.  \frac{1}{2}\sigma_{i}^{2}(s,\tilde{X}%
_{s})\partial_{xx}^{2}f(s,\tilde{X}_{s},\phi^{\ast}(s,\tilde{X}_{s}%
))\right \vert \right]  ds\\
&  \text{ \  \  \ }+\int_{t_{i}}^{t}\mathbb{E}\left[  \left \vert \partial
_{t}f(s,\tilde{X}_{s},\phi^{\ast}(t_{i},\tilde{X}_{t_{i}}))+b_{i}(s,\tilde
{X}_{s})\partial_{x}f(s,\tilde{X}_{s},\phi^{\ast}(t_{i},\tilde{X}_{t_{i}%
}))\right.  \right.  \nonumber \\
&  \text{ \  \  \ }+\left.  \left.  \left.  \frac{1}{2}\sigma_{i}^{2}%
(s,\tilde{X}_{s})\partial_{xx}^{2}f(s,\tilde{X}_{s},\phi^{\ast}(t_{i}%
,\tilde{X}_{t_{i}}))\right \vert \right]  ds\right \}  dt.\nonumber
\end{align}
Then, under the conditions of $b,\sigma,\phi^{\ast}$\ and $f$, by Lemma
\ref{SDE_Estimate}, we obtain
\begin{equation}
\left \vert J_{2}\right \vert \leq C\Delta t+C\sum_{i=0}^{N-1}\int_{t_{i}%
}^{t_{i+1}}\int_{t_{i}}^{t}\mathbb{E}\left[  1+|\tilde{X}_{s}|^{2}\right]
dsdt\leq C\Delta t.
\end{equation}
In conclusion, $\left(  \ref{20}\right)  $ holds. The proof is complete.
\end{proof}
\end{theorem}

\subsection{Recursive approximation}

In this subsection, we estimate the recursive approximation error $\left \vert
J\left(  \bar{u}\right)  -J\left(  u^{N}\right)  \right \vert $ generated by
the numerical recursive approximation. Consider the following FBSDEs:%
\begin{equation}
\left \{
\begin{array}
[c]{l}%
\displaystyle X_{t}^{N}=X_{t_{i}}^{N}+\int_{t_{i}}^{t}b\left(  s,X_{s}%
^{N},\phi_{i}^{N}(X_{t_{i}}^{N})\right)  ds+\int_{t_{i}}^{t}\sigma \left(
s,X_{s}^{N},\phi_{i}^{N}(X_{t_{i}}^{N})\right)  dW_{s},\\
\displaystyle P_{t}^{N}=P_{t_{i+1}}^{N}+\int_{t}^{t_{i+1}}H_{x}\left(
s,X_{s}^{N},P_{s}^{N},Q_{s}^{N},\phi_{i}^{N}(X_{t_{i}}^{N})\right)
ds-\int_{t}^{t_{i+1}}Q_{s}^{N}dW_{s},
\end{array}
\right.  \label{23}%
\end{equation}
for $t\in \lbrack t_{i},t_{i+1}]$, $i=0,1,\ldots,N-1$, with $X_{t}^{N}=x_{0}$
and $P_{T}^{N}=h_{x}(X_{T}^{N})$. For notational simplicity, in the sequel,
for $t\in \lbrack t_{i},t_{i+1}]$, we let
\[
\mathcal{H}_{t}^{N}=H_{x}\left(  t,X_{t}^{N},P_{t}^{N},Q_{t}^{N},\phi_{i}%
^{N}(X_{t_{i}}^{N})\right)  .
\]

Let $(\bar{X}_{t}^{t_{i},\alpha},\bar{P}_{t}^{t_{i},\alpha},\bar{Q}_{t}%
^{t_{i},\alpha})$ be the solution of FBSDEs $\left(  \ref{10.1}\right)
-\left(  \ref{10.2}\right)  $ with $\bar{X}_{t_{i}}=\alpha$, for $t\in \lbrack
t_{i},T]$. Denote $b(t,\bar{X}_{t}^{t_{i},\alpha},\bar{\phi}_{i}(\alpha))$ by
$\bar{b}_{t}^{N,t_{i},\alpha}$, $\sigma(t,\bar{X}_{t}^{t_{i},\alpha}$,
$\bar{\phi}_{i}(\alpha))$ by $\bar{\sigma}_{t}^{N,t_{i},\alpha}$ and
\[
\mathcal{\bar{H}}_{t}^{t_{i},\alpha}=H_{x}\left(  t,\bar{X}_{t}^{t_{i},\alpha
},\bar{P}_{t}^{t_{i},\alpha},\bar{Q}_{t}^{t_{i},\alpha},\bar{\phi}_{i}\left(
\alpha \right)  \right)  .
\]
Then for $i=0,1,\ldots,N-1$,
\begin{equation}
\left \{
\begin{array}
[c]{l}%
\displaystyle \bar{X}_{t_{i+1}}^{t_{i},X_{t_{i}}^{N}}=X_{t_{i}}^{N}%
+\int_{t_{i}}^{t_{i+1}}\bar{b}_{t}^{N,t_{i},X_{t_{i}}^{N}}dt+\int_{t_{i}%
}^{t_{i+1}}\bar{\sigma}_{t}^{N,t_{i},X_{t_{i}}^{N}}dW_{t},\\
\displaystyle \bar{P}_{t_{i}}^{t_{i},X_{t_{i}}^{N}}=\bar{P}_{t_{i+1}}%
^{t_{i},X_{t_{i}}^{N}}+\int_{t_{i}}^{t_{i+1}}\mathcal{\bar{H}}_{t}%
^{t_{i},X_{t_{i}}^{N}}dt-\int_{t_{i}}^{t_{i+1}}\bar{Q}_{t}^{t_{i},X_{t_{i}%
}^{N}}dW_{t}.
\end{array}
\right.  \label{29}%
\end{equation}

We have the following error estimate.

\begin{lemma}
\label{lemma6}Suppose $\left(  A1\right)  \ $holds. Let $\left(  \bar{X}%
_{t},\bar{P}_{t},\bar{Q}_{t}\right)  $ and $\left(  X_{t}^{N},P_{t}^{N}%
,Q_{t}^{N}\right)  $, $t\in \left[  0,T\right]  $, be the solutions of $\left(
\ref{10.1}\right)  -\left(  \ref{10.2}\right)  $ and $\left(  \ref{23}\right)
$, respectively. We also assume $b,\sigma,f\in C_{b}^{2,5,5}$, $\{ \bar{\phi
}_{i}\}_{i=1}^{N-1},\{ \phi_{i}^{N}\}_{i=1}^{N-1}\in C_{b}^{4}$ and $h\in
C_{b}^{5+\alpha}$, $\alpha>0$, and there exists a positive constant $L$, not
depending on $N$, such that $\sup_{i}\left(  |\phi_{i}^{N}(0)|+|\bar{\phi}%
_{i}(0)|\right)  \leq L$. Then for $i=0,1,\ldots,N-1$,
\[
\mathbb{E}\left[  \left \vert \bar{P}_{t_{i}}^{t_{i},X_{t_{i}}^{N}}-P_{t_{i}%
}^{N}\right \vert ^{2}\right]  +\Delta t\sum_{j=i}^{N-1}\mathbb{E}\left[
\left \vert \bar{Q}_{t_{j}}^{t_{j},X_{t_{j}}^{N}}-Q_{t_{j}}^{N}\right \vert
^{2}\right]  \leq C\left(  \Delta t\right)  ^{2}+C\Delta t\sum_{j=i}%
^{N-1}\mathbb{E}\left[  \left \vert \bar{\phi}_{j}(X_{t_{j}}^{N})-\phi_{j}%
^{N}(X_{t_{j}}^{N})\right \vert ^{2}\right]  .
\]

\begin{proof}
For simplicity, we introduce the following notation:
\[%
\begin{array}
[c]{ll}%
\hat{P}_{t_{i}}=\bar{P}_{t_{i}}^{t_{i},X_{t_{i}}^{N}}-P_{t_{i}}^{N},\text{ } &
\hat{P}_{t_{i+1}}=\bar{P}_{t_{i+1}}^{t_{i+1},X_{t_{i+1}}^{N}}-P_{t_{i+1}}%
^{N},\\
\hat{Q}_{t_{i}}=\bar{Q}_{t_{i}}^{t_{i},X_{t_{i}}^{N}}-Q_{t_{i}}^{N},\text{ } &
\hat{Q}_{t_{i+1}}=\bar{Q}_{t_{i+1}}^{t_{i+1},X_{t_{i+1}}^{N}}-Q_{t_{i+1}}^{N}.
\end{array}
\]
We also denote $\mathcal{\bar{H}}_{t}^{t_{i},X_{t_{i}}^{N}}-\mathcal{H}%
_{t}^{N}$ by $\mathcal{\hat{H}}_{t}$. For each integer $0\leq i\leq N-1,$ from
$\left(  \ref{23}\right)  $ and $\left(  \ref{29}\right)  $, we obtain%
\begin{equation}
\hat{P}_{t_{i}}=\hat{P}_{t_{i+1}}+\bar{P}_{t_{i+1}}^{t_{i},X_{t_{i}}^{N}}%
-\bar{P}_{t_{i+1}}^{t_{i+1},X_{t_{i+1}}^{N}}+\int_{t_{i}}^{t_{i+1}%
}\mathcal{\hat{H}}_{t}dt-\int_{t_{i}}^{t_{i+1}}\left(  \bar{Q}_{t}%
^{t_{i},X_{t_{i}}^{N}}-Q_{t}^{N}\right)  dW_{t}. \label{45}%
\end{equation}
Define the conditional mathematical expectation $\mathbb{E}_{t_{i}}^{x}\left[
\cdot \right]  :=\mathbb{E}\left[  \left.  \cdot \right \vert \bar{X}_{t_{i}%
}=x\right]  $ and denote $\mathbb{E}_{t_{i}}^{X_{t_{i}}^{N}}\left[
\cdot \right]  $ by $\mathbb{E}_{t_{i}}^{N}\left[  \cdot \right]  $. It is easy
to check that the equation above\ is equivalent to the following equations:%
\begin{align}
&  \hat{P}_{t_{i}}=\mathbb{E}_{t_{i}}^{N}[\hat{P}_{t_{i+1}}]+\mathcal{\hat{H}%
}_{t_{i}}\Delta t+R_{p,i}+\bar{R}_{p,i},\label{p}\\
&  \hat{Q}_{t_{i}}\Delta t=\mathbb{E}_{t_{i}}^{N}[\hat{P}_{t_{i+1}}\Delta
W_{t_{i+1}}]+\bar{R}_{q,i}+R_{q,i}, \label{q}%
\end{align}
where the error terms
\begin{align*}
&
\begin{array}
[c]{l}%
\displaystyle R_{p,i}=\int_{t_{i}}^{t_{i+1}}\mathbb{E}_{t_{i}}^{N}\left[
\mathcal{\hat{H}}_{t}-\mathcal{\hat{H}}_{t_{i}}\right]  dt,\\
\displaystyle R_{q,i}=\int_{t_{i}}^{t_{i+1}}\mathbb{E}_{t_{i}}^{N}\left[
\mathcal{\hat{H}}_{t}\Delta W_{t_{i+1}}\right]  dt-\int_{t_{i}}^{t_{i+1}%
}\mathbb{E}_{t_{i}}^{N}\left[  \left(  \bar{Q}_{t}^{t_{i},X_{t_{i}}^{N}}%
-Q_{t}^{N}\right)  -\hat{Q}_{t_{i}}\right]  dt,
\end{array}
\\
&
\begin{array}
[c]{ll}%
\displaystyle \bar{R}_{p,i}=\mathbb{E}_{t_{i}}^{N}\left[  \bar{P}_{t_{i+1}%
}^{t_{i},X_{t_{i}}^{N}}-\bar{P}_{t_{i+1}}^{t_{i+1},X_{t_{i+1}}^{N}}\right]
, & \displaystyle \bar{R}_{q,i}=\mathbb{E}_{t_{i}}^{N}\left[  \left(  \bar
{P}_{t_{i+1}}^{t_{i},X_{t_{i}}^{N}}-\bar{P}_{t_{i+1}}^{t_{i+1},X_{t_{i+1}}%
^{N}}\right)  \Delta W_{t_{i+1}}\right]  .
\end{array}
\end{align*}
By H\"{o}lder's inequality, we have
\begin{align}
\left \vert \mathbb{E}_{t_{i}}^{N}[\hat{P}_{t_{i+1}}\Delta W_{t_{i+1}%
}]\right \vert ^{2}  &  \leq \left \vert \mathbb{E}_{t_{i}}^{N}\left[  \left(
\hat{P}_{t_{i+1}}-\mathbb{E}_{t_{i}}^{N}[\hat{P}_{t_{i+1}}]\right)  \Delta
W_{t_{i+1}}\right]  \right \vert ^{2}\label{46}\\
&  \leq \left(  \mathbb{E}_{t_{i}}^{N}\left[  |\hat{P}_{t_{i+1}}|^{2}\right]
-\left \vert \mathbb{E}_{t_{i}}^{N}[\hat{P}_{t_{i+1}}]\right \vert ^{2}\right)
\Delta t.\nonumber
\end{align}
Taking square of $\left(  \ref{p}\right)  -\left(  \ref{q}\right)  $ and using
the inequalities $\left(  \ref{46}\right)  $ and $(a+b)^{2}\leq(1+\gamma \Delta
t)a^{2}+(1+\frac{1}{\gamma \Delta t})b^{2}$, we obtain
\begin{align*}
|\hat{P}_{t_{i}}|^{2}  &  \leq \left(  1+\gamma \Delta t\right)  \left \vert
\mathbb{E}_{t_{i}}^{N}[\hat{P}_{t_{i+1}}]\right \vert ^{2}+C(1+\frac{1}%
{\gamma \Delta t})\left(  \left \vert R_{p,i}\right \vert ^{2}+\left \vert \bar
{R}_{p,i}\right \vert ^{2}\right) \\
&  \text{ \  \  \ }+C(1+\frac{1}{\gamma \Delta t})\Delta t^{2}\left(  |\hat
{P}_{t_{i}}|^{2}+|\hat{Q}_{t_{i}}|^{2}+\left \vert \bar{\phi}_{i}(X_{t_{i}}%
^{N})-\phi_{i}^{N}(X_{t_{i}}^{N})\right \vert ^{2}\right)  ,\\
|\hat{Q}_{t_{i}}|^{2}  &  \leq \frac{C}{\Delta t}\left(  \mathbb{E}_{t_{i}}%
^{N}\left[  |\hat{P}_{t_{i+1}}|^{2}\right]  -\left \vert \mathbb{E}_{t_{i}}%
^{N}[\hat{P}_{t_{i+1}}]\right \vert ^{2}\right)  +\frac{C}{(\Delta t)^{2}%
}\left(  \left \vert R_{q,i}\right \vert ^{2}+\left \vert \bar{R}_{q,i}%
\right \vert ^{2}\right)  .
\end{align*}
By choosing $\gamma=2C^{2}$, $\gamma \Delta t\leq1$, and adding up the above
inequalities, we obtain
\begin{align*}
|\hat{P}_{t_{i}}|^{2}+\frac{\Delta t}{2C}|\hat{Q}_{t_{i}}|^{2}  &  \leq \left(
1+2C^{2}\Delta t\right)  \mathbb{E}_{t_{i}}^{N}\left[  |\hat{P}_{t_{i+1}}%
|^{2}\right]  +\frac{\Delta t}{2C}\left(  |\hat{P}_{t_{i}}|^{2}+\left \vert
\bar{\phi}_{i}(X_{t_{i}}^{N})-\phi_{i}^{N}(X_{t_{i}}^{N})\right \vert
^{2}\right) \\
&  \text{ \  \  \ }+\frac{1}{2C\Delta t}\left(  \left \vert R_{p,i}\right \vert
^{2}+\left \vert \bar{R}_{p,i}\right \vert ^{2}+\left \vert R_{q,i}\right \vert
^{2}+\left \vert \bar{R}_{q,i}\right \vert ^{2}\right)  ,
\end{align*}
which yields%
\begin{align}
|\hat{P}_{t_{i}}|^{2}+C\Delta t|\hat{Q}_{t_{i}}|^{2}  &  \leq \left(  1+C\Delta
t\right)  \mathbb{E}_{t_{i}}^{N}\left[  |\hat{P}_{t_{i+1}}|^{2}\right]
+C\Delta t\left \vert \bar{\phi}_{i}(X_{t_{i}}^{N})-\phi_{i}^{N}(X_{t_{i}}%
^{N})\right \vert ^{2}\label{48}\\
&  \text{ \  \  \ }+\frac{C}{\Delta t}\left(  \left \vert R_{p,i}\right \vert
^{2}+\left \vert \bar{R}_{p,i}\right \vert ^{2}+\left \vert R_{q,i}\right \vert
^{2}+\left \vert \bar{R}_{q,i}\right \vert ^{2}\right)  .\nonumber
\end{align}
Taking mathematical expectation on both sides of $\left(  \ref{48}\right)  ,$
we have%
\begin{align}
\mathbb{E}\left[  |\hat{P}_{t_{i}}|^{2}\right]  +C\Delta t\mathbb{E}\left[
|\hat{Q}_{t_{i}}|^{2}\right]   &  \leq \left(  1+C\Delta t\right)
\mathbb{E}\left[  |\hat{P}_{t_{i+1}}|^{2}\right]  +C\Delta t\mathbb{E}\left[
\left \vert \bar{\phi}_{i}(X_{t_{i}}^{N})-\phi_{i}^{N}(X_{t_{i}}^{N}%
)\right \vert ^{2}\right] \label{49}\\
&  \text{ \  \  \ }+\frac{C}{\Delta t}\mathbb{E}\left[  \left \vert
R_{p,i}\right \vert ^{2}+\left \vert \bar{R}_{p,i}\right \vert ^{2}+\left \vert
R_{q,i}\right \vert ^{2}+\left \vert \bar{R}_{q,i}\right \vert ^{2}\right]
,\nonumber
\end{align}
which, by induction, leads to the inequality%
\begin{align}
\mathbb{E}\left[  |\hat{P}_{t_{i}}|^{2}\right]   &  \leq C\mathbb{E}\left[
|\hat{P}_{t_{N}}|^{2}\right]  +C\Delta t\sum_{j=i}^{N-1}\mathbb{E}\left[
\left \vert \bar{\phi}_{j}(X_{t_{j}}^{N})-\phi_{j}^{N}(X_{t_{j}}^{N}%
)\right \vert ^{2}\right] \\
&  \text{ \  \  \ }+\frac{C}{\Delta t}\sum_{j=i}^{N-1}\mathbb{E}\left[
\left \vert R_{p,j}\right \vert ^{2}+\left \vert \bar{R}_{p,j}\right \vert
^{2}+\left \vert R_{q,j}\right \vert ^{2}+\left \vert \bar{R}_{q,j}\right \vert
^{2}\right]  .\nonumber
\end{align}
Based on $\left(  \ref{49}\right)  $, we can deduce
\begin{align}
C\Delta t\sum_{j=i}^{N-1}\mathbb{E}\left[  |\hat{Q}_{t_{j}}|^{2}\right]   &
\leq C\Delta t\sum_{j=i}^{N-1}\mathbb{E}\left[  |\hat{P}_{t_{j+1}}%
|^{2}\right]  +C\sum_{j=i}^{N-1}\mathbb{E}\left[  \left \vert \bar{\phi}%
_{j}(X_{t_{j}}^{N})-\phi_{j}^{N}(X_{t_{j}}^{N})\right \vert ^{2}\right]  \Delta
t\\
&  \text{ \  \  \ }+\frac{C}{\Delta t}\sum_{j=i}^{N-1}\mathbb{E}\left[
\left \vert R_{p,j}\right \vert ^{2}+\left \vert \bar{R}_{p,j}\right \vert
^{2}+\left \vert R_{q,j}\right \vert ^{2}+\left \vert \bar{R}_{q,j}\right \vert
^{2}\right] \nonumber \\
\text{ \  \  \  \  \ }  &  \leq C\mathbb{E}\left[  |\hat{P}_{t_{N}}|^{2}\right]
+C\sum_{j=i}^{N-1}\mathbb{E}\left[  \left \vert \bar{\phi}_{j}(X_{t_{j}}%
^{N})-\phi_{j}^{N}(X_{t_{j}}^{N})\right \vert ^{2}\right]  \Delta t\nonumber \\
&  \text{ \  \  \ }+\frac{C}{\Delta t}\sum_{j=i}^{N-1}\mathbb{E}\left[
\left \vert R_{p,j}\right \vert ^{2}+\left \vert \bar{R}_{p,j}\right \vert
^{2}+\left \vert R_{q,j}\right \vert ^{2}+\left \vert \bar{R}_{q,j}\right \vert
^{2}\right]  .\nonumber
\end{align}
Thus%
\begin{align}
\mathbb{E}\left[  |\hat{P}_{t_{i}}|^{2}\right]  +\Delta t\sum_{j=i}%
^{N-1}\mathbb{E}\left[  |\hat{Q}_{t_{j}}|^{2}\right]   &  \leq C\mathbb{E}%
\left[  |\hat{P}_{t_{N}}|^{2}\right]  +C\sum_{j=i}^{N-1}\mathbb{E}\left[
\left \vert \bar{\phi}_{j}(X_{t_{j}}^{N})-\phi_{j}^{N}(X_{t_{j}}^{N}%
)\right \vert ^{2}\right]  \Delta t\label{51}\\
&  \text{ \  \ }+\frac{C}{\Delta t}\sum_{j=i}^{N-1}\mathbb{E}\left[  \left \vert
R_{p,j}\right \vert ^{2}+\left \vert \bar{R}_{p,j}\right \vert ^{2}+\left \vert
R_{q,j}\right \vert ^{2}+\left \vert \bar{R}_{q,j}\right \vert ^{2}\right]
.\nonumber
\end{align}
Now we estimate the error terms $R_{p,j},R_{q,j},\bar{R}_{p,j}$ and $\bar
{R}_{q,j}$. For $\left[  t_{j},t_{j+1}\right]  ,j=N-1,\ldots,i$, under the
condition of Lemma \ref{lemma6}, similar to the variational method in Lemma
\ref{relation lemma}, the solutions of $\left(  \ref{29}\right)  $ on $\left[
t_{j},t_{j+1}\right]  $\ have the representations $\bar{P}_{t}^{t_{j},x}%
=m_{j}(t,\bar{X}_{t}^{t_{j},x})\in C_{b}^{2,4}\ $and $\bar{Q}_{t}^{t_{j}%
,x}=g_{j}(t,\bar{X}_{t}^{t_{j},x})\in C_{b}^{2,4}$. By It\^{o}'s formula, we
can deduce%
\begin{align*}
\left \vert \bar{R}_{p,j}\right \vert  &  =\left \vert \mathbb{E}_{t_{j}}%
^{N}\left[  m_{j}\left(  t_{j+1},\bar{X}_{t_{j+1}}^{t_{j},X_{t_{j}}^{N}%
}\right)  -m_{j}\left(  t_{j},X_{t_{j}}^{N}\right)  \right]  \right. \\
&  \text{ \  \  \ }-\left.  \mathbb{E}_{t_{j}}^{N}\left[  m_{j}\left(
t_{j+1},X_{t_{j+1}}^{N}\right)  -m_{j}\left(  t_{j},X_{t_{j}}^{N}\right)
\right]  \right \vert \\
&  =\left \vert \int_{t_{j}}^{t_{j+1}}\mathbb{E}_{t_{j}}^{N}\left[
\mathcal{L}m_{j}\left(  t_{j},X_{t_{j}}^{N}\right)  +\int_{t_{j}}%
^{t}\mathcal{LL}m_{j}\left(  s,\bar{X}_{s}^{t_{j},X_{t_{j}}^{N}}\right)
ds\right]  dt\right. \\
&  \text{ \  \  \ }-\left.  \int_{t_{j}}^{t_{j+1}}\mathbb{E}_{t_{j}}^{N}\left[
\mathcal{\tilde{L}}m_{j}\left(  t_{j},X_{t_{j}}^{N}\right)  +\int_{t_{j}}%
^{t}\mathcal{\tilde{L}\tilde{L}}m_{j}\left(  s,X_{s}^{N}\right)  ds\right]
dt\right \vert \\
&  \leq \mathbb{E}_{t_{j}}^{N}\left[  \left \vert \mathcal{L}m_{j}\left(
t_{j},X_{t_{j}}^{N}\right)  -\mathcal{\tilde{L}}m_{j}\left(  t_{j},X_{t_{j}%
}^{N}\right)  \right \vert \right]  \Delta t\\
&  \text{ \  \  \ }+C\int_{t_{j}}^{t_{j+1}}\int_{t_{j}}^{t}\left(
1+\mathbb{E}_{t_{j}}^{N}\left[  |X_{s}^{N}|^{4}+|X_{t_{j}}^{N}|^{4}\right]
\right)  dsdt,
\end{align*}
where
\begin{align*}
\mathcal{L}  &  =\frac{\partial}{\partial t}+\sum \limits_{k=1}^{n}b_{k}\left(
t,x,\bar{\phi}_{j}(X_{t_{j}}^{N})\right)  \frac{\partial}{\partial x_{k}%
}+\frac{1}{2}\sum \limits_{k,l=1}^{n}\left[  \sigma \sigma^{\top}\right]
_{k,l}\left(  t,x,\bar{\phi}_{j}(X_{t_{j}}^{N})\right)  \frac{\partial^{2}%
}{\partial x_{k}\partial x_{l}},\\
\mathcal{\tilde{L}}  &  =\frac{\partial}{\partial t}+\sum \limits_{k=1}%
^{n}b_{k}\left(  t,x,\phi_{j}^{N}(X_{t_{j}}^{N})\right)  \frac{\partial
}{\partial x_{k}}+\frac{1}{2}\sum \limits_{k,l=1}^{n}\left[  \sigma \sigma
^{\top}\right]  _{k,l}\left(  t,x,\phi_{j}^{N}(X_{t_{j}}^{N})\right)
\frac{\partial^{2}}{\partial x_{k}\partial x_{l}}.
\end{align*}
Then by Lemma \ref{SDE_Estimate}, we have%
\begin{equation}
\mathbb{E}\left[  \left \vert \bar{R}_{p,j}\right \vert ^{2}\right]  \leq
C\left(  \Delta t\right)  ^{2}\mathbb{E}\left[  \left \vert \bar{\phi}%
_{j}(X_{t_{j}}^{N})-\phi_{j}^{N}(X_{t_{j}}^{N})\right \vert ^{2}\right]
+C\left(  \Delta t\right)  ^{4}.
\end{equation}
Similarly,
\begin{equation}
\mathbb{E}\left[  \left \vert \bar{R}_{q,j}\right \vert ^{2}\right]  \leq
C\left(  \Delta t\right)  ^{2}\mathbb{E}\left[  \left \vert \bar{\phi}%
_{j}(X_{t_{j}}^{N})-\phi_{j}^{N}(X_{t_{j}}^{N})\right \vert ^{2}\right]
+C\left(  \Delta t\right)  ^{4}.
\end{equation}
In the same way above, we can obtain
\begin{equation}%
\begin{array}
[c]{ll}%
\mathbb{E}[\left \vert R_{p,j}\right \vert ^{2}]\leq C\left(  \Delta t\right)
^{4},\text{ \ } & \mathbb{E}[\left \vert R_{q,j}\right \vert ^{2}]\leq C\left(
\Delta t\right)  ^{4}.
\end{array}
\label{52}%
\end{equation}
Consequently, the desired conclusion follows from $\left(  \ref{51}\right)
-\left(  \ref{52}\right)  .$
\end{proof}
\end{lemma}

Now we discuss the error produced by the numerical solution of FBSDEs in
Schemes \ref{Scheme1}$-$\ref{Scheme2}. It is easy to check that solving the
FBSDEs $\left(  \ref{23}\right)  $\ is equivalent to finding the solution to
the following equations:%
\begin{align}
P_{t_{i}}^{N}  &  =\mathbb{E}_{t_{i}}\left[  P_{t_{i+1}}^{N}\right]
+\mathcal{H}_{t_{i}}^{N}\Delta t+\mathcal{E}_{P,i},\label{P}\\
Q_{t_{i}}^{N}  &  =\left(  \mathbb{E}_{t_{i}}\left[  P_{t_{i+1}}^{N}\Delta
W_{t_{i+1}}\right]  +\mathcal{E}_{Q,i}\right)  /\Delta t, \label{Q}%
\end{align}
where $\mathbb{E}_{t_{i}}\left[  \cdot \right]  =\mathbb{E}\left[  \left.
\cdot \right \vert \mathcal{F}_{t_{i}}\right]  $ and the truncation errors%
\[%
\begin{array}
[c]{l}%
\displaystyle \mathcal{E}_{P,i}=\int_{t_{i}}^{t_{i+1}}\mathbb{E}_{t_{i}%
}\left[  \mathcal{H}_{t}^{N}\right]  dt-\mathcal{H}_{t_{i}}^{N}\Delta t,\\
\displaystyle \mathcal{E}_{Q,i}=\int_{t_{i}}^{t_{i+1}}\mathbb{E}_{t_{i}%
}\left[  \mathcal{H}_{t}^{N}\Delta W_{t_{i+1}}\right]  dt-\int_{t_{i}%
}^{t_{i+1}}\mathbb{E}_{t_{i}}\left[  Q_{t}^{N}\right]  dt+Q_{t_{i}}^{N}\Delta
t.
\end{array}
\]

\begin{lemma}
\label{lemma5}Suppose $\left(  A1\right)  $ and the conditions in Lemma
\ref{lemma6} hold. Let $(P_{i}^{N,t_{i},x},Q_{i}^{N,t_{i},x})$ be the
numerical solution at grid point $\left(  t_{i},x\right)  $. We assume that
$\{P_{i}^{N}\}_{i=1}^{N-1}\in C_{b}^{4}$. Then for $i=0,1,\ldots,N-1$,%
\[
\mathbb{E}\left[  \left \vert P_{t_{i}}^{N}-P_{i}^{N,t_{i},X_{t_{i}}^{N}%
}\right \vert ^{2}\right]  +\Delta t\sum_{j=i}^{N-1}\mathbb{E}\left[
\left \vert Q_{t_{j}}^{N}-Q_{j}^{N,t_{j},X_{t_{j}}^{N}}\right \vert ^{2}\right]
\leq C\left(  \Delta t\right)  ^{2}.
\]

\end{lemma}

\begin{proof}
The proof of lemma can be referred to \cite{ZhZhJu2014,GLTZhZh2017}.
\end{proof}

Based on the above discussion, we have the following lemma.

\begin{lemma}
\label{lemma4}Suppose $\left(  A2\right)  $ and the conditions in Lemmas
\ref{lemma6}$-$\ref{lemma5} hold. Then for $i=0,1,\ldots,N-1,$
\[
\Delta t\sum_{j=i}^{N-1}\mathbb{E}\left[  \left \vert \bar{\phi}_{j}(X_{t_{j}%
}^{N})-\phi_{j}^{N}(X_{t_{j}}^{N})\right \vert ^{2}\right]  \leq C\left(
\Delta t\right)  ^{2}.
\]

\begin{proof}
Provided that $X_{t_{i}}^{N}=x$, from $\left(  \ref{Theorem 1.2}\right)  $ we
know
\[
\left \langle \bar{\phi}_{i}(x)-\frac{\rho}{\Delta t}\int_{t_{i}}^{t_{i+1}%
}\mathbb{E}\left[  H_{u}\left(  t,\bar{X}_{t}^{t_{i},x},\bar{P}_{t}^{t_{i}%
,x},\bar{Q}_{t}^{t_{i},x},\bar{\phi}_{i}(x)\right)  \right]  dt-\bar{\phi}%
_{i}(x),\phi_{i}(x)-\bar{\phi}_{i}(x)\right \rangle \leq0\text{,}%
\]
for any $\phi_{i}(\cdot)\in C_{b}^{1}(\mathbb{R}^{n};U)$ and $\rho>0$, which
implies
\begin{equation}
\bar{\phi}_{i}(x)=P_{U}\left(  \bar{\phi}_{i}(x)-\frac{\rho}{\Delta t}%
\int_{t_{i}}^{t_{i+1}}\mathbb{E}\left[  H_{u}\left(  t,\bar{X}_{t}^{t_{i}%
,x},\bar{P}_{t}^{t_{i},x},\bar{Q}_{t}^{t_{i},x},\bar{\phi}_{i}(x)\right)
\right]  dt\right)  , \label{35}%
\end{equation}
where $P_{U}\ $is the projection operator from $\mathbb{R}^{n}$ to $U$, such
that%
\[
P_{U}\left(  v\right)  =\arg \min_{u\in U}\left \vert u-v\right \vert ^{2}.
\]
Analogously, from $\left(  \ref{fai^N}\right)  $, we have
\begin{equation}
\phi_{i}^{N}(x)=P_{U}\left(  \phi_{i}^{N}(x)-\rho H_{u}\left(  t_{i}%
,x,P_{i}^{N,t_{i},x},Q_{i}^{N,t_{i},x},\phi_{i}^{N}(x)\right)  \right)
,\text{ }\rho>0. \label{36}%
\end{equation}
For convenience, we omit the superscript $^{t_{i},x}$ if no ambiguity arises.
From $\left(  \ref{35}\right)  -\left(  \ref{36}\right)  $, it is easy to
obtain
\begin{align}
&  \; \; \; \; \left \vert \bar{\phi}_{i}(x)-\phi_{i}^{N}(x)\right \vert
\label{39}\\
&  \leq \left \vert \bar{\phi}_{i}(x)-\phi_{i}^{N}(x)-\rho \left[  H_{u}\left(
t_{i},x,P_{t_{i}}^{N},Q_{t_{i}}^{N},\phi_{i}^{N}(x)\right)  -H_{u}\left(
t_{i},x,P_{i}^{N},Q_{i}^{N},\phi_{i}^{N}(x)\right)  \right]  \right.
\nonumber \\
&  \text{\  \  \ }-\left.  \rho \left[  H_{u}\left(  t_{i},x,\bar{P}_{t_{i}}%
,\bar{Q}_{t_{i}},\bar{\phi}_{i}(x)\right)  -H_{u}\left(  t_{i},x,P_{t_{i}}%
^{N},Q_{t_{i}}^{N},\phi_{i}^{N}(x)\right)  \right]  -\rho R_{H}^{i}\right \vert
,\nonumber
\end{align}
where%
\[
R_{H}^{i}=\frac{1}{\Delta t}\int_{t_{i}}^{t_{i+1}}\mathbb{E}\left[
H_{u}\left(  t,\bar{X}_{t},\bar{P}_{t},\bar{Q}_{t},\bar{\phi}_{i}(x)\right)
-H_{u}\left(  t_{i},x,\bar{P}_{t_{i}},\bar{Q}_{t_{i}},\bar{\phi}%
_{i}(x)\right)  \right]  dt.
\]
Then, we have%
\begin{align}
&  \; \; \; \; \left \vert \bar{\phi}_{i}(x)-\phi_{i}^{N}(x)\right \vert
^{2}\label{37}\\
&  \leq \left \vert \bar{\phi}_{i}(x)-\phi_{i}^{N}(x)\right \vert ^{2}+\rho
^{2}\left \vert R_{H}^{i}\right \vert ^{2}-2\rho \left \langle \bar{\phi}%
_{i}(x)-\phi_{i}^{N}(x),R_{H}^{i}\right \rangle \nonumber \\
&  \text{ \  \ }+\rho^{2}\left \vert H_{u}\left(  t_{i},x,\bar{P}_{t_{i}}%
,\bar{Q}_{t_{i}},\bar{\phi}_{i}(x)\right)  -H_{u}\left(  t_{i},x,P_{t_{i}}%
^{N},Q_{t_{i}}^{N},\phi_{i}^{N}(x)\right)  \right \vert ^{2}\text{
\  \  \  \  \  \  \  \  \  \  \ }\nonumber \\
&  \text{ \  \ }+\rho^{2}\left \vert H_{u}\left(  t_{i},x,P_{t_{i}}^{N}%
,Q_{t_{i}}^{N},\phi_{i}^{N}(x)\right)  -H_{u}\left(  t_{i},x,P_{i}^{N}%
,Q_{i}^{N},\phi_{i}^{N}(x)\right)  \right \vert ^{2}\nonumber
\end{align}%
\begin{align*}
&  \text{ \  \ }-2\rho \left \langle H_{u}\left(  t_{i},x,\bar{P}_{t_{i}},\bar
{Q}_{t_{i}},\bar{\phi}_{i}(x)\right)  -H_{u}\left(  t_{i},x,P_{t_{i}}%
^{N},Q_{t_{i}}^{N},\phi_{i}^{N}(x)\right)  ,\bar{\phi}_{i}(x)-\phi_{i}%
^{N}(x)\right \rangle \\
&  \text{ \  \ }-2\rho \left \langle H_{u}\left(  t_{i},x,P_{t_{i}}^{N},Q_{t_{i}%
}^{N},\phi_{i}^{N}(x)\right)  -H_{u}\left(  t_{i},x,P_{i}^{N},Q_{i}^{N}%
,\phi_{i}^{N}(x)\right)  ,\bar{\phi}_{i}(x)-\phi_{i}^{N}(x)\right \rangle \\
&  \text{ \  \ }+2\rho^{2}\left \langle H_{u}\left(  t_{i},x,\bar{P}_{t_{i}%
},\bar{Q}_{t_{i}},\bar{\phi}_{i}(x)\right)  -H_{u}\left(  t_{i},x,P_{t_{i}%
}^{N},Q_{t_{i}}^{N},\phi_{i}^{N}(x)\right)  \right.  ,\\
&  \text{ \  \  \  \  \  \  \  \  \ }\left.  H_{u}\left(  t_{i},x,P_{t_{i}}%
^{N},Q_{t_{i}}^{N},\phi_{i}^{N}(x)\right)  -H_{u}\left(  t_{i},x,P_{i}%
^{N},Q_{i}^{N},\phi_{i}^{N}(x)\right)  \right \rangle \\
&  \text{ \  \ }+2\rho^{2}\left \langle H_{u}\left(  t_{i},x,\bar{P}_{t_{i}%
},\bar{Q}_{t_{i}},\bar{\phi}_{i}(x)\right)  -H_{u}\left(  t_{i},x,P_{t_{i}%
}^{N},Q_{t_{i}}^{N},\phi_{i}^{N}(x)\right)  ,R_{H}^{i}\right \rangle \\
&  \text{ \  \ }+2\rho^{2}\left \langle H_{u}\left(  t_{i},x,P_{t_{i}}%
^{N},Q_{t_{i}}^{N},\phi_{i}^{N}(x)\right)  -H_{u}\left(  t_{i},x,P_{i}%
^{N},Q_{i}^{N},\phi_{i}^{N}(x)\right)  ,R_{H}^{i}\right \rangle .
\end{align*}
Notice that
\begin{equation}
\text{ \  \  \ }%
\begin{array}
[c]{l}%
\text{ \ }-2\rho \left \langle H_{u}\left(  t_{i},x,P_{t_{i}}^{N},Q_{t_{i}}%
^{N},\phi_{i}^{N}(x)\right)  -H_{u}\left(  t_{i},x,P_{i}^{N},Q_{i}^{N}%
,\phi_{i}^{N}(x)\right)  ,\bar{\phi}_{i}(x)-\phi_{i}^{N}(x)\right \rangle \\
\leq \rho^{2}\left \vert H_{u}\left(  t_{i},x,P_{t_{i}}^{N},Q_{t_{i}}^{N}%
,\phi_{i}^{N}(x)\right)  -H_{u}\left(  t_{i},x,P_{i}^{N},Q_{i}^{N},\phi
_{i}^{N}(x)\right)  \right \vert ^{2}+\left \vert \bar{\phi}_{i}(x)-\phi_{i}%
^{N}(x)\right \vert ^{2}.
\end{array}
\label{38}%
\end{equation}
Similarly,%
\begin{equation}%
\begin{array}
[c]{l}%
\text{ \ }2\rho^{2}\left \langle H_{u}\left(  t_{i},x,\bar{P}_{t_{i}},\bar
{Q}_{t_{i}},\bar{\phi}_{i}(x)\right)  -H_{u}\left(  t_{i},x,P_{t_{i}}%
^{N},Q_{t_{i}}^{N},\phi_{i}^{N}(x)\right)  ,R_{H}^{i}\right \rangle \\
\leq \rho^{2}\left \vert H_{u}\left(  t_{i},x,\bar{P}_{t_{i}},\bar{Q}_{t_{i}%
},\bar{\phi}_{i}(x)\right)  -H_{u}(t_{i},x,P_{t_{i}}^{N},Q_{t_{i}}^{N}%
,\phi_{i}^{N}(x))\right \vert ^{2}+\rho^{2}\left \vert R_{H}^{i}\right \vert
^{2},
\end{array}
\text{ \  \ }%
\end{equation}%
\begin{equation}%
\begin{array}
[c]{l}%
\text{ \ }2\rho^{2}\left \langle H_{u}\left(  t_{i},x,P_{t_{i}}^{N},Q_{t_{i}%
}^{N},\phi_{i}^{N}(x)\right)  -H_{u}\left(  t_{i},x,P_{i}^{N},Q_{i}^{N}%
,\phi_{i}^{N}(x)\right)  ,R_{H}^{i}\right \rangle \\
\leq \rho^{2}\left \vert H_{u}\left(  t_{i},x,P_{t_{i}}^{N},Q_{t_{i}}^{N}%
,\phi_{i}^{N}(x)\right)  -H_{u}\left(  t_{i},x,P_{i}^{N},Q_{i}^{N},\phi
_{i}^{N}(x)\right)  \right \vert ^{2}+\rho^{2}\left \vert R_{H}^{i}\right \vert
^{2},
\end{array}
\end{equation}%
\begin{equation}%
\begin{array}
[c]{l}%
\text{\  \ }2\rho^{2}\left \langle H_{u}\left(  t_{i},x,\bar{P}_{t_{i}},\bar
{Q}_{t_{i}},\bar{\phi}_{i}(x)\right)  -H_{u}\left(  t_{i},x,P_{t_{i}}%
^{N},Q_{t_{i}}^{N},\phi_{i}^{N}(x)\right)  \right.  ,\\
\text{ \  \  \  \  \  \ }\left.  H_{u}\left(  t_{i},x,P_{t_{i}}^{N},Q_{t_{i}}%
^{N},\phi_{i}^{N}(x)\right)  -H_{u}\left(  t_{i},x,P_{i}^{N},Q_{i}^{N}%
,\phi_{i}^{N}(x)\right)  \right \rangle \\
\leq \rho^{2}\left \vert H_{u}\left(  t_{i},x,\bar{P}_{t_{i}},\bar{Q}_{t_{i}%
},\bar{\phi}_{i}(x)\right)  -H_{u}\left(  t_{i},x,P_{t_{i}}^{N},Q_{t_{i}}%
^{N},\phi_{i}^{N}(x)\right)  \right \vert ^{2}\\
\text{ \ }+\rho^{2}\left \vert H_{u}\left(  t_{i},x,P_{t_{i}}^{N},Q_{t_{i}}%
^{N},\phi_{i}^{N}(x)\right)  -H_{u}\left(  t_{i},x,P_{i}^{N},Q_{i}^{N}%
,\phi_{i}^{N}(x)\right)  \right \vert ^{2},
\end{array}
\text{ \  \  \  \  \  \  \  \  \  \ }%
\end{equation}
and%
\begin{equation}
-2\rho \left \langle \bar{\phi}_{i}(x)-\phi_{i}^{N}(x),R_{H}\right \rangle
\leq \rho^{2}\left \vert \bar{\phi}_{i}(x)-\phi_{i}^{N}(x)\right \vert
^{2}+\left \vert R_{H}^{i}\right \vert ^{2}.\text{
\  \  \  \  \  \  \  \  \  \  \  \  \  \  \  \  \ } \label{40}%
\end{equation}
From $\left(  \ref{37}\right)  -\left(  \ref{40}\right)  $ and assumption
$\left(  A2\right)  $, we have%
\begin{align}
\left \vert \bar{\phi}_{i}(x)-\phi_{i}^{N}(x)\right \vert ^{2}  &  \leq \left(
1-2c_{0}\rho+8C\rho^{2}\right)  \left \vert \bar{\phi}_{i}(x)-\phi_{i}%
^{N}(x)\right \vert ^{2}\label{42}\\
&  \text{ \  \ }+2C\left(  1+4\rho^{2}\right)  \left(  \left \vert P_{t_{i}}%
^{N}-P_{i}^{N}\right \vert ^{2}+\left \vert Q_{t_{i}}^{N}-Q_{i}^{N}\right \vert
^{2}\right) \nonumber \\
&  \text{ \  \ }+9C\rho^{2}\left(  \left \vert \bar{P}_{t_{i}}-P_{t_{i}}%
^{N}\right \vert ^{2}+\left \vert \bar{Q}_{t_{i}}-Q_{t_{i}}^{N}\right \vert
^{2}\right) \nonumber \\
&  \text{ \  \ }+\left(  1+3\rho^{2}\right)  \left \vert R_{H}^{i}\right \vert
^{2}.\nonumber
\end{align}
Choosing sufficiently small $\rho$ in $\left(  \ref{42}\right)  $, such that
$2c_{0}\rho-8C\rho^{2}\geq c_{0}\rho/2$, we obtain%
\begin{align}
\left \vert \bar{\phi}_{i}(x)-\phi_{i}^{N}(x)\right \vert ^{2}  &  \leq
\frac{C\rho}{c_{0}}\left(  \left \vert \bar{P}_{t_{i}}-P_{t_{i}}^{N}\right \vert
^{2}+\left \vert \bar{Q}_{t_{i}}-Q_{t_{i}}^{N}\right \vert ^{2}\right)  \text{
\  \  \  \  \  \  \  \  \  \  \ }\label{43}\\
&  \text{ \  \ }+\frac{C}{c_{0}\rho}\left(  \left \vert P_{t_{i}}^{N}-P_{i}%
^{N}\right \vert ^{2}+\left \vert Q_{t_{i}}^{N}-Q_{i}^{N}\right \vert ^{2}\right)
\nonumber \\
&  \text{ \  \ }+\frac{C}{c_{0}\rho}\left \vert R_{H}^{i}\right \vert
^{2}.\nonumber
\end{align}
Then we can deduce%
\begin{align*}
\Delta t\sum_{j=i}^{N-1}\mathbb{E}\left[  \left \vert \bar{\phi}_{j}(X_{t_{j}%
}^{N})-\phi_{j}^{N}(X_{t_{j}}^{N})\right \vert ^{2}\right]   &  \leq \frac
{C}{c_{0}\rho}\Delta t\sum_{j=i}^{N-1}\mathbb{E}\left[  \left \vert P_{t_{j}%
}^{N}-P_{j}^{N,t_{j},X_{t_{j}}^{N}}\right \vert ^{2}+\left \vert Q_{t_{j}}%
^{N}-Q_{j}^{N,t_{j},X_{t_{j}}^{N}}\right \vert ^{2}\right] \\
&  \text{ \  \ }+\frac{C\rho}{c_{0}}\Delta t\sum_{j=i}^{N-1}\mathbb{E}\left[
\left \vert \bar{P}_{t_{j}}^{t_{j},X_{t_{j}}^{N}}-P_{t_{j}}^{N}\right \vert
^{2}+\left \vert \bar{Q}_{t_{j}}^{t_{j},X_{t_{j}}^{N}}\,-Q_{t_{j}}%
^{N}\right \vert ^{2}\right]  \text{ \  \  \  \  \  \ }\\
&  \text{ \  \ }+\frac{C}{c_{0}\rho}\Delta t\sum_{j=i}^{N-1}\mathbb{E}\left[
\left \vert R_{H}^{i}\right \vert ^{2}\right]  .
\end{align*}
By It\^{o}'s formula, it is easy to check $\mathbb{E}\left[  |R_{H}^{i}%
|^{2}\right]  \leq C\left(  \Delta t\right)  ^{2}$. Then, by Lemmas
\ref{lemma6} and \ref{lemma5}, we have
\begin{equation}
\Delta t\sum_{j=i}^{N-1}\mathbb{E}\left[  \left \vert \bar{\phi}_{j}(X_{t_{j}%
}^{N})-\phi_{j}^{N}(X_{t_{j}}^{N})\right \vert ^{2}\right]  \leq \frac{C\rho
}{c_{0}}\Delta t\sum_{j=i}^{N-1}\mathbb{E}\left[  \left \vert \bar{\phi}%
_{j}(X_{t_{j}}^{N})-\phi_{j}^{N}(X_{t_{j}}^{N})\right \vert ^{2}\right]
+\frac{C}{c_{0}\rho}\left(  1+\rho^{2}\right)  \left(  \Delta t\right)
^{2}.\nonumber
\end{equation}
Further choosing the constant $\rho$, such that $C\rho/c_{0}\leq1/2$, the
desired result follows.
\end{proof}
\end{lemma}

\begin{theorem}
\label{theorem3}Suppose $\left(  A2\right)  $ and the conditions in Lemmas
\ref{lemma6}$-$\ref{lemma5} hold. Then
\[
\left \vert J\left(  \bar{u}\right)  -J\left(  u^{N}\right)  \right \vert \leq
C\Delta t.
\]

\begin{proof}
First, we rewrite the state equations $\left(  \ref{4}\right)  $ and $\left(
\ref{27}\right)  \ $as follows:%
\begin{align*}
&  \bar{X}_{t}=x_{0}+\int_{0}^{t}\bar{b}\left(  s,\bar{X}_{s}\right)
ds+\int_{0}^{t}\bar{\sigma}\left(  s,\bar{X}_{s}\right)  dW_{s},\\
&  X_{t}^{N}=x_{0}+\int_{0}^{t}\tilde{b}\left(  s,X_{s}^{N}\right)
ds+\int_{0}^{t}\tilde{\sigma}\left(  s,X_{s}^{N}\right)  dW_{s},
\end{align*}
where\ for $s\in \left[  t_{i},t_{i+1}\right]  $ $\left(  i=0,\ldots
,N-1\right)  ,$%
\[%
\begin{array}
[c]{ll}%
\bar{b}\left(  s,x\right)  =b\left(  s,x,\bar{\phi}_{i}(\bar{X}_{t_{i}%
})\right)  , & \bar{\sigma}\left(  s,x\right)  =\sigma \left(  s,x,\bar{\phi
}_{i}(\bar{X}_{t_{i}})\right)  ,\\
\tilde{b}\left(  s,x\right)  =b\left(  s,x,\phi_{i}^{N}(X_{t_{i}}^{N})\right)
, & \tilde{\sigma}\left(  s,x\right)  =\sigma \left(  s,x,\phi_{i}^{N}%
(X_{t_{i}}^{N})\right)  .
\end{array}
\]
By Lemmas \ref{lemma2} and \ref{lemma4}, for $0\leq t\leq T$, we have
\begin{align}
&  \; \; \; \; \mathbb{E}\left[  \sup_{0\leq s\leq t}|\bar{X}_{s}-X_{s}%
^{N}|^{2}\right] \\
&  \leq C\int_{0}^{t}\mathbb{E}\left[  \left(  \bar{b}\left(  s,\bar{X}%
_{s}\right)  -\tilde{b}\left(  s,\bar{X}_{s}\right)  \right)  ^{2}+\left(
\bar{\sigma}\left(  s,\bar{X}_{s}\right)  -\tilde{\sigma}\left(  s,\bar{X}%
_{s}\right)  \right)  ^{2}\right]  ds\nonumber \\
&  \leq C\int_{0}^{t}\sum_{i=0}^{\tau_{t}}\mathbb{E}\left[  \left(  \bar{\phi
}_{i}(\bar{X}_{t_{i}})-\phi_{i}^{N}(X_{t_{i}}^{N})\right)  ^{2}\right]
I_{\left[  t_{i},t_{i+1}\right)  }(s)ds\nonumber \\
&  \leq C\sum_{i=0}^{\tau_{t}}\left \{  \mathbb{E}\left[  \left(  \bar{\phi
}_{i}(X_{t_{i}}^{N})-\phi_{i}^{N}(X_{t_{i}}^{N})\right)  ^{2}\right]
+\mathbb{E}\left[  \left(  \bar{X}_{t_{i}}-X_{t_{i}}^{N}\right)  ^{2}\right]
\right \}  \Delta t\nonumber \\
&  \leq C\int_{0}^{t}\mathbb{E}\left[  \sup_{0\leq r\leq s}|\bar{X}_{r}%
-X_{r}^{N}|^{2}\right]  ds+C\left(  \Delta t\right)  ^{2},\nonumber
\end{align}
where $\tau_{t}$ is an integer, satisfying $t_{\tau_{t}}<t\leq t_{\tau_{t}+1}%
$. Then, by Gronwall's inequality, we obtain
\begin{equation}
\mathbb{E}\left[  \sup_{0\leq s\leq t}|\bar{X}_{s}-X_{s}^{N}|^{2}\right]  \leq
C\left(  \Delta t\right)  ^{2}, \label{32}%
\end{equation}
for $0\leq t\leq T$. Notice that
\begin{align}
&  \; \; \; \; \left \vert J\left(  \bar{u}\right)  -J\left(  u^{N}\right)
\right \vert \\
&  \leq \sum_{i=0}^{N-1}\int_{t_{i}}^{t_{i+1}}\mathbb{E}\left[  \left \vert
f\left(  t,\bar{X}_{t},\bar{\phi}_{i}(\bar{X}_{t_{i}})\right)  -f\left(
t,X_{t}^{N},\phi_{i}^{N}(X_{t_{i}}^{N})\right)  \right \vert \right]
dt\nonumber \\
&  \text{ \  \ }+\mathbb{E}\left[  \left \vert h\left(  \bar{X}_{T}\right)
-h(X_{T}^{N})\right \vert \right] \nonumber \\
&  \leq \sum_{i=0}^{N-1}\int_{t_{i}}^{t_{i+1}}\mathbb{E}\left[  \left \vert
f\left(  t,\bar{X}_{t},\bar{\phi}_{i}(\bar{X}_{t_{i}})\right)  -f\left(
t,X_{t}^{N},\bar{\phi}_{i}(X_{t_{i}}^{N})\right)  \right \vert \right]
dt\nonumber \\
&  \text{ \  \ }+\sum_{i=0}^{N-1}\int_{t_{i}}^{t_{i+1}}\mathbb{E}\left[
\left \vert f\left(  t,X_{t}^{N},\bar{\phi}_{i}(X_{t_{i}}^{N})\right)
-f\left(  t,X_{t}^{N},\phi_{i}^{N}(X_{t_{i}}^{N})\right)  \right \vert \right]
dt\nonumber \\
&  \text{ \  \ }+\mathbb{E}\left[  \left \vert h(\bar{X}_{T})-h(X_{T}%
^{N})\right \vert \right]  .\nonumber
\end{align}
Since the continuity of $f$ and $h$, by H\"{o}lder's inequality, we have%
\begin{align}
&  \; \; \; \; \left \vert J\left(  \bar{u}\right)  -J\left(  u^{N}\right)
\right \vert \label{34}\\
&  \leq C\sum_{i=0}^{N-1}\int_{t_{i}}^{t_{i+1}}\left[  \left(  \mathbb{E}%
\left[  |\bar{X}_{t}-X_{t}^{N}|^{2}\right]  \right)  ^{\frac{1}{2}}+\left(
\mathbb{E}\left[  |\bar{X}_{t_{i}}-X_{t_{i}}^{N}|^{2}\right]  \right)
^{\frac{1}{2}}\right]  dt\nonumber \\
&  \text{ \  \ }+C\sum_{i=0}^{N-1}\mathbb{E}\left[  \left \vert \bar{\phi}%
_{i}\left(  X_{t_{i}}^{N}\right)  -\phi_{i}^{N}(X_{t_{i}}^{N})\right \vert
\right]  \Delta t+C\left(  \mathbb{E}\left[  |\bar{X}_{T}-X_{T}^{N}%
|^{2}\right]  \right)  ^{\frac{1}{2}}\nonumber \\
&  \leq C\left(  \mathbb{E}\left[  \sup_{0\leq t\leq T}|\bar{X}_{t}-X_{t}%
^{N}|^{2}\right]  \right)  ^{\frac{1}{2}}+C\sum_{i=0}^{N-1}\mathbb{E}\left[
\left \vert \bar{\phi}_{i}\left(  X_{t_{i}}^{N}\right)  -\phi_{i}^{N}(X_{t_{i}%
}^{N})\right \vert \right]  \Delta t.\nonumber
\end{align}
In addition, by H\"{o}lder's inequality, we have
\begin{equation}
\Delta t\sum_{i=0}^{N-1}\mathbb{E}\left[  \left \vert \bar{\phi}_{i}\left(
X_{t_{i}}^{N}\right)  -\phi_{i}^{N}(X_{t_{i}}^{N})\right \vert \right]
\leq \sqrt{T}\left(  \sum_{i=0}^{N-1}\mathbb{E}\left[  \left \vert \bar{\phi
}_{i}\left(  X_{t_{i}}^{N}\right)  -\phi_{i}^{N}(X_{t_{i}}^{N})\right \vert
^{2}\right]  \Delta t\right)  ^{\frac{1}{2}}. \label{33}%
\end{equation}
Combining $\left(  \ref{32}\right)  ,\left(  \ref{34}\right)  -\left(
\ref{33}\right)  $ and Lemma \ref{lemma4}, we complete the proof.
\end{proof}
\end{theorem}

\subsection{Proof of the main results \label{proof of theorem}}

\begin{proof}
[Proof of Theorem \ref{theorem} ]Notice that
\begin{align*}
\left \vert J\left(  u^{\ast}\right)  -J\left(  u^{N}\right)  \right \vert  &
=\left \vert J\left(  u^{\ast}\right)  -J\left(  \bar{u}\right)  +J\left(
\bar{u}\right)  -J\left(  u^{N}\right)  \right \vert \\
&  \leq \left \vert J\left(  u^{\ast}\right)  -J\left(  \bar{u}\right)
\right \vert +\left \vert J\left(  \bar{u}\right)  -J\left(  u^{N}\right)
\right \vert .
\end{align*}
By Theorems \ref{theorem2} and \ref{theorem3}, we complete our proof.
\end{proof}

\section{Numerical experiments}

In this section, some numerical experiments have been presented to illustrate
the high accuracy of our algorithm for solving SOCPs. The first two examples
are deterministic control, and the latter two examples are feedback control.
In our tests, we use Gauss-Hermite quadrature rule to approximate the
conditional mathematical expectation and use cubic spline interpolation to
compute spatial non-grid points. To make sure the first-order convergence of
our method, the Euler method is also adopted to solve the related FBSDEs when
calculating the cost. In the following tables, CR stands for the convergence rate.

\begin{example}
\label{EX1}We first consider the control problem of the Black-Scholes type in
\cite{DSL2013}
\[
\left \{
\begin{array}
[c]{l}%
dX_{t}=u_{t}X_{t}dt+\sigma X_{t}dW_{t},\\
X_{0}=x_{0},
\end{array}
\right.
\]
with the cost functional
\[
J\left(  u\right)  =\frac{1}{2}\int_{0}^{T}\mathbb{E}\left[  \left(
X_{t}-\eta_{t}^{\ast}\right)  ^{2}\right]  dt+\frac{1}{2}\int_{0}^{T}u_{t}%
^{2}dt.\text{\ }%
\]
The function $\eta_{\cdot}^{\ast}$ and the corresponding optimal control
$u_{\cdot}^{\ast}$ can be expressed as%
\begin{equation}%
\begin{array}
[c]{ll}%
\displaystyle \eta_{t}^{\ast}=\frac{e^{\sigma^{2}t}-\left(  T-t\right)  ^{2}%
}{\frac{1}{x_{0}}-Tt+\frac{t^{2}}{2}}+1,\text{ \ } & \displaystyle u_{t}%
^{\ast}=\frac{T-t}{x_{0}-Tt+\frac{t^{2}}{2}}.
\end{array}
\tag{a}%
\end{equation}
We set $x_{0}=1,T=1\ $and $\sigma=0.1$ and the reference optimal cost is
$J\left(  u^{\ast}\right)  =0.514898066090988$. Numerical results by using our
discrete recursive method are listed in Table \ref{Table1}.
\begin{table}[ptbh]
\caption{Errors and convergence rates for Example \ref{EX1} a.}%
\label{Table1}
{\footnotesize
\[%
\begin{tabular}
[c]{|c|c|c|c|c|c|c|}\hline
$N$ & 8 & 16 & 32 & 64 & 128 & CR\\ \hline
\multicolumn{1}{|l|}{$\left \vert J\left(  u^{\ast}\right)  -J\left(
u^{N}\right)  \right \vert $} & \multicolumn{1}{|l|}{1.393E-01} &
\multicolumn{1}{|l|}{1.364E-01} & \multicolumn{1}{|l|}{8.512E-02} &
\multicolumn{1}{|l|}{3.243E-02} & 9.068E-03 & 0.996\\ \hline
\end{tabular}
\]
}\end{table}

Next, we choose a different $\eta_{\cdot}^{\ast}$ and $u_{\cdot}^{\ast},$
which is%
\begin{equation}%
\begin{array}
[c]{ll}%
\displaystyle \eta_{t}^{\ast}=\frac{e^{\sigma^{2}t}-\left(  e^{-T}%
-e^{-t}\right)  ^{2}}{\frac{1}{x_{0}}+1-e^{-t}-te^{-T}}-e^{-t}, &
\displaystyle u_{t}^{\ast}=\frac{e^{-T}-e^{-t}}{\frac{1}{x_{0}}+1-e^{-t}%
-te^{-T}}.
\end{array}
\tag{b}%
\end{equation}
Set $x_{0}=1,T=1\ $and $\sigma=0.1$. The reference optimal cost is $J\left(
u^{\ast}\right)  =0.345819897539892$. Numerical results in Table \ref{Table2}
demonstrate that our method is stable and admits a first order rate of
convergence. \begin{table}[ptbh]
\caption{Errors and convergence rates for Example \ref{EX1} b.}%
\label{Table2}
{\footnotesize
\[%
\begin{tabular}
[c]{|c|c|c|c|c|c|c|}\hline
$N$ & 8 & 16 & 32 & 64 & 128 & CR\\ \hline
\multicolumn{1}{|l|}{$\left \vert J\left(  u^{\ast}\right)  -J\left(
u^{N}\right)  \right \vert $} & \multicolumn{1}{|l|}{5.931E-02} &
\multicolumn{1}{|l|}{2.826E-02} & \multicolumn{1}{|l|}{1.369E-02} &
\multicolumn{1}{|l|}{6.554E-03} & 3.056E-03 & 1.067\\ \hline
\end{tabular}
\]
}\end{table}
\end{example}

\begin{example}
\label{EX2}
The second example is the inventory control problem in \cite{DSL2013}. The
inventory level satisfies the following equation
\[
\left \{
\begin{array}
[c]{l}%
dX_{t}=\left(  u_{t}-r_{t}\right)  dt+\sigma dW_{t},\\
X_{0}=x_{0},
\end{array}
\right.
\]
with the total cost%
\[
J\left(  u\right)  =\frac{1}{2}\int_{0}^{T}\mathbb{E}\left[  \left(
X_{t}-\eta_{t}\right)  ^{2}\right]  dt+\frac{1}{2}\int_{0}^{T}u_{t}%
^{2}dt.\text{\ }%
\]
The demand rates $r_{t}=\left(  T-t\right)  /2\ $and $\eta_{t}=0.5Tt-0.25t^{2}%
+1$. Then the optimal production $u_{\cdot}^{\ast}$ and the optimal cost can
be expressed as
\[%
\begin{array}
[c]{rr}%
\displaystyle u_{t}^{\ast}=T-t,\text{ } & \displaystyle J\left(  u^{\ast
}\right)  =\frac{1}{6}T^{3}+\frac{\sigma^{2}-2}{4}T^{2}+T.
\end{array}
\]
We set $x_{0}=0$ and $T=1$. Table \ref{Table3} shows the numerical results of
the Example \ref{EX2} with $\sigma=0.0,$ $0.1$ and $0.3$, respectively. It
clearly shows that the cost obtained by our numerical method admits a first
order rate of convergence. \begin{table}[ptbh]
\caption{Errors and convergence rates for Example \ref{EX2}.}%
\label{Table3}
{\footnotesize
\[%
\begin{tabular}
[c]{||c|c|c|c|c|c|c||}\hline
$N$ & 8 & 16 & 32 & 64 & 128 & \\ \hline
& \multicolumn{5}{|c|}{$\left \vert J\left(  u^{\ast}\right)  -J\left(
u^{N}\right)  \right \vert $} & CR\\ \hline
\multicolumn{1}{||l|}{$\sigma=0.0$} & \multicolumn{1}{|l|}{5.654E-02} &
\multicolumn{1}{|l|}{2.746E-02} & \multicolumn{1}{|l|}{1.380E-02} &
\multicolumn{1}{|l|}{6.870E-03} & \multicolumn{1}{|l|}{3.532E-03} &
1.000\\ \hline
$\sigma=0.1$ & 7.888E-02 & 4.408E-02 & 2.286E-02 & 1.108E-02 & 4.908E-03 &
1.000\\ \hline
$\sigma=0.3$ & 1.321E-01 & 8.608E-02 & 5.204E-02 & 2.548E-02 & 7.563E-03 &
1.001\\ \hline
\end{tabular}
\]
}\end{table}
\end{example}

\begin{example}
\label{EX3}The third example is a LQ problem in \cite{YZh1999}
\[
\left \{
\begin{array}
[c]{l}%
dX_{t}=u_{t}dt+\delta u_{t}dW_{t},\\
X_{0}=x_{0},
\end{array}
\right.
\]
with the cost functional%
\[
J\left(  u\right)  =\frac{1}{2}\int_{0}^{T}\mathbb{E}\left[  X_{t}^{2}\right]
dt.\text{\ }%
\]
The optimal control and the corresponding optimal cost are given by
\[%
\begin{array}
[c]{rr}%
u_{t}^{\ast}=-\frac{X_{t}}{\delta^{2}},\text{ \ } & J\left(  u^{\ast}\right)
=\frac{1}{2}\delta^{2}\left(  1-e^{\frac{-T}{\delta^{2}}}\right)  .
\end{array}
\]
Set $x_{0}=1,T=1\ $and $\delta=2$. Numerical results are listed in Table
\ref{Table5}. It is clearly shown that our method is stable and admits a first
order rate of convergence. \begin{table}[ptbh]
\caption{Errors and convergence rates for Example \ref{EX3}.}%
\label{Table5}%
{\footnotesize
\[%
\begin{tabular}
[c]{|c|c|c|c|c|c|c|}\hline
$N$ & 8 & 16 & 32 & 64 & 128 & CR\\ \hline
\multicolumn{1}{|l|}{$\left \vert J\left(  u^{\ast}\right)  -J\left(
u^{N}\right)  \right \vert $} & 9.611E-03 & 4.653E-03 &
\multicolumn{1}{|l|}{2.338E-03} & \multicolumn{1}{|l|}{1.193E-03} &
\multicolumn{1}{|l|}{6.114E-04} & 0.991\\ \hline
\end{tabular}
\]
}\end{table}
\end{example}

\begin{example}
\label{EX4}In last example we consider a portfolio problem
\begin{equation}
\left \{
\begin{array}
[c]{l}%
\displaystyle d\tilde{X}_{t}=\left(  \alpha \tilde{u}_{t}+\gamma \right)
\tilde{X}_{t}dt+\beta \tilde{u}_{t}\tilde{X}_{t}d\widetilde{W}_{t},\text{ }%
t\in(0,1],\\
\displaystyle \tilde{X}_{0}=\tilde{x}_{0},
\end{array}
\right.  \label{exp3.1}%
\end{equation}
with the cost functional%
\[
\tilde{J}\left(  u^{\ast}\right)  =\min_{u\in K}\frac{1}{2}\mathbb{E}\left[
(\tilde{X}_{1}-\kappa)^{2}\right]  ,\label{exp3.2}%
\]
and%
\[
K=\left \{  \tilde{u}_{\cdot}\in \mathcal{U}[0,1]:-1\leq \tilde{u}_{t}%
\leq1,\text{ }a.e.\text{ }a.s.\right \}  .
\]
Set $\tilde{x}_{0}=6,\kappa=20,\alpha=0.25$, $\gamma=1$ and $\beta=\sqrt{2}%
/2$. The reference optimal cost with a fine mesh is $\tilde{J}\left(  u^{\ast
}\right)  =6.00909101172000$. Since the control set $U=\left[  -1,1\right]  $,
we need to project $\tilde{u}_{t}$ into $\left[  -1,1\right]  $ by $\tilde
{u}_{t}=\max \left(  \min \left(  \tilde{u}_{t},1\right)  ,-1\right)  .$
We remark that the Bisection method can also be used to solve this example.
Numerical results are listed in Table \ref{Table4}. It is clearly shown that
our method admits a first order rate of convergence.
\begin{table}[ptbh]
\caption{Errors and convergence rates for Example \ref{EX4}.}%
\label{Table4}%
{\footnotesize
\[%
\begin{tabular}
[c]{|c|c|c|c|c|c|c|}\hline
$N$ & 8 & 16 & 32 & 64 & 128 & CR\\ \hline
\multicolumn{1}{|l|}{$|\tilde{J}\left(  u^{\ast}\right)  -\tilde{J}\left(
u^{N}\right)  |$} & 3.592E+00 & 1.797E+00 & \multicolumn{1}{|l|}{9.761E-01} &
\multicolumn{1}{|l|}{4.622E-01} & \multicolumn{1}{|l|}{2.205E-01} &
1.001\\ \hline
\end{tabular}
\]
}\end{table}

\begin{remark}
The control problem above is obtained from Example 4 in \cite{GLTZhZh2017}
through the following transformation
\[
\tilde{X}_{t}=\frac{1}{T}X_{Tt},\text{ \  \ }\tilde{u}_{t}=u_{Tt},\text{
\ }\tilde{J}\left(  \tilde{u}\right)  =\frac{1}{T^{2}}J\left(  u\right)  ,
\]
and the process $\widetilde{W}_{t}=\frac{1}{\sqrt{T}}W_{Tt}$ with the $\sigma
$-field $\mathcal{F}_{t}^{\widetilde{W}}=\mathcal{F}_{Tt}^{W}.$
\end{remark}
\end{example}

\section{Conclusion}

In this work, we reduce the optimal control problem to the discrete case and
derive a discrete SMP. By means of this discrete SMP, we propose an effective
discrete recursive method for solving SOCPs. The Euler scheme is used to
approximate the discrete Hamilton system that is given by the discrete SMP
condition and the state and adjoint equations. We conducted a rigorous error
analysis and prove that our method admits a first order rate of convergence.
Several numerical examples powerful support the theoretical results.

\end{document}